\documentclass[reqno]{amsart}

\usepackage[english]{babel}
\usepackage{amssymb,amsmath,amscd,amsfonts,amsthm,bbm,color}
\usepackage{hyperref}
\usepackage{graphicx}

\newcommand{\tQ}{\widetilde Q}
\newcommand{\tD}{ D}
\newcommand{\tT}{\widetilde T}

\DeclareMathOperator{\var}{\mathbb Var}

\DeclareMathOperator{\rank}{rank}

\newcommand{\1}{\mathbbm 1}

\newcommand{\N}{\mathbb N}
\newcommand{\R}{\mathbb R}

\newcommand{\C}{\mathbb{C}}

\newcommand{\E}{\mathbb E}

\newcommand{\cal}{\mathcal}
    
\theoremstyle{plain} 

\newtheorem{theorem}{Theorem}[section]
\newtheorem{lemma}{Lemma}[section]
\newtheorem{corollary}{Corollary}[section]

\newtheorem{proposition}{Proposition}[section]
\newtheorem{assumption}{Assumption}
\newtheorem{remark}{Remark}

\DeclareMathOperator{\Tr}{Tr}

\DeclareMathOperator*{\adjugate}{adj}
\DeclareMathOperator*{\support}{supp}
\newcommand{\PP}{{\mathbb{P}}}
\newcommand{\bs}{\boldsymbol}
\newcommand{\EE}{{\mathbb{E}}}
\DeclareMathOperator*{\diag}{diag}
\newcommand{\toasshort}{\stackrel{\text{as}}{\to}}
\newcommand{\toaslong}{\xrightarrow[n\to\infty]{\text{a.s.}}}

\newcommand{\toprobashort}{\,\stackrel{\mathcal{P}}{\to}\,}
\newcommand{\toprobalong}{\xrightarrow[n\to\infty]{\mathcal P}}
\DeclareMathOperator{\tr}{Tr}

\newcommand{\dmax}{{\bf d}_{\max}} 
\newcommand{\pmax}{{\bf p}_{\max}} 
\newcommand{\mfat}{{\bf m}} 

\title[Singular value outliers]
{The outliers among the singular values \\
of large rectangular random matrices \\ 
with additive fixed rank deformation} 

\author[Chapon, Couillet, Hachem and Mestre]
{Fran\c cois Chapon, Romain Couillet, \\ Walid Hachem and
Xavier Mestre}

\thanks{The work of the first three authors was supported by the French
Ile-de-France region, DIM LSC fund, Digiteo project DESIR} 

\begin{document}

\begin{abstract}
Consider the matrix 
$\Sigma_n = n^{-1/2} X_n D_n^{1/2} + P_n$ where the matrix 
$X_n \in \C^{N\times n}$ has Gaussian 
standard independent elements, $D_n$ is a deterministic diagonal
nonnegative matrix, and $P_n$ is a deterministic matrix with fixed rank. 
Under some known conditions, the spectral measures of $\Sigma_n \Sigma_n^*$
and $n^{-1} X_n D_n X_n^*$ both converge towards a compactly supported 
probability
measure $\mu$ as $N,n\to\infty$ with $N/n\to c>0$.
In this paper, it is proved that finitely many eigenvalues of
$\Sigma_n\Sigma_n^*$ may stay away from the support of $\mu$ in the large
dimensional regime.
The existence and locations of these outliers in any connected component 
of $\R - \support(\mu)$ are studied. 
The fluctuations of the largest outliers of $\Sigma_n\Sigma_n^*$ 
are also analyzed. The results find applications in the fields of signal 
processing and radio communications. 
\end{abstract}



\date{January 2012} 

\maketitle

\section{Introduction}
\label{sec:intro}

\subsection{The model and the literature}
\label{subsec:model}
Consider a sequence of $N \times n$ matrices $Y_n$, $n=1,2,\ldots$, of the form 
$Y_n=X_n D_n^{1/2}$ where $X_n$ is a $N\times n$ random matrix 
whose coefficients $X_{ij}$ are independent and identically distributed (iid)
complex Gaussian random variables such that $\Re(X_{11})$ and $\Im(X_{11})$
are independent, each with mean zero and variance $1/2$, and where 
$D_n$ is a deterministic nonnegative diagonal $n\times n$ matrix. 
Writing $D_n = \diag( d_j^n)_{j=1,\ldots,n}$ and denoting by 
$\bs\delta$ the Dirac measure, it is assumed that the spectral 
measure $\nu_n = n^{-1} \sum_{j=1}^n \bs\delta_{ d_j^n}$ of 
$D_n$ converges weakly to a compactly supported probability 
measure $\nu$ when $n\to\infty$. 
It is also assumed that the maximum of the distances from the diagonal elements 
of $D_n$ to the support $\support(\nu)$ of $\nu$ goes 
to zero as $n\to\infty$.
Assume that $N/n \to c$ when $n\to\infty$, where $c$ is a positive constant. 
Then it is known that 
with probability one, the spectral measure of the Gram matrix $n^{-1}Y_n Y_n^*$ 
converges weakly to a compactly supported probability measure $\mu$
(see \cite{MarPas67}, \cite{Gir90}, \cite{Sil95}, \cite{SilBai95}) and, 
with probability one, $n^{-1}Y_n Y_n^*$ has no eigenvalues 
in any compact interval outside $\support(\mu)$ for large $n$ \cite{BaiSil98}. 
\\
Let $r$ be a given positive integer and consider a sequence of 
deterministic $N \times n$ matrices $P_n$, $n=1,2,\ldots$, such that $\rank(P_n) = r$ and 
$\sup_n \| P_n \| < \infty$ where $\| \cdot \|$ is the spectral norm. 
Consider the matrix $\Sigma_n = n^{-1/2} Y_n + P_n$. 
Since the additive deformation $P_n$ has a 
fixed rank, the spectral measure of $\Sigma_n \Sigma_n^*$ still converges 
to $\mu$ (see, \emph{e.g.}, \cite[Lemma 2.2]{Bai99}). However, a finite 
number of eigenvalues of $\Sigma_n\Sigma_n^*$ (often called ``outliers'' in
similar contexts) may stay away of the support of $\mu$. 
In this paper, minimal conditions ensuring the existence and the convergence
of these outliers towards constant values outside $\support(\mu)$ are provided,
and these limit values are characterized. The fluctuations of the outliers
lying at the right of $\support(\mu)$ are also studied. \\ 

The behavior of the outliers in the spectrum of large random matrices has
aroused an important research effort.  In the statistics literature,
one of the first contributions to deal with this subject was \cite{joh01}. It
raised the question of the behavior of the extreme eigenvalues of a sample
covariance matrix when the population covariance matrix has all but finitely
many of its eigenvalues equal to one (leading to a mutliplicative fixed rank
deformation).  This problem has been studied thoroughly in
\cite{bbp05,bk-sil06,paul-07}.  Other contributions (see \cite{cdf09}) study
the outliers of a Wigner matrix subject to an additive fixed rank deformation.
The asymptotic fluctuations of the outliers have been addressed in
\cite{bbp05,peche-06,paul-07,by08,cdf09,cdf-ihp12,bgm-fluct10}. 

Recently, Benaych-Georges and Nadakuditi proposed in
\cite{ben-rao-published, ben-rao-11} a generic method for
characterizing the behavior of the outliers for a large palette of
random matrix models. For our model, this method shows that the limiting 
locations 
as well as the fluctuations of the outliers are intimately related 
to the asymptotic behavior of certain bilinear forms involving the resolvents 
$(n^{-1} Y_n Y_n^* - x I_N)^{-1}$ and 
$(n^{-1} Y_n^* Y_n - x I_n)^{-1}$ of the undeformed matrix for real
values of $x$. 
When $D_n = I_n$, the asymptotic behavior of these bilinear forms 
can be simply identified (see \cite{ben-rao-11}) thanks to the fact that  
the probability law of $Y_n$ is invariant by left or right multiplication by 
deterministic unitary matrices. For general $D_n$, other tools
need to be used. In this paper, these bilinear forms are studied with the help 
of an integration by parts formula for functionals of Gaussian vectors and the 
Poincar\'e-Nash inequality. 
These tools belong to the arsenal of random matrix theory, as shown  
in the recent monograph \cite{pas-livre} and in the references therein. 
In order to be able to use them in our context, we make use of a regularizing 
function ensuring that the moments of the bilinear forms exist for certain 
$x \in \R_+ = [0, \infty)$.
\\  

The study of the spectrum outliers of large random matrices has a wide range of
applications. These include communication theory \cite{hlmnv-subspace11}, fault
diagnosis in complex systems \cite{cou-hac-it12}, financial portfolio
management \cite{pot-bouch-lal-05}, or chemometrics \cite{chemo}.  The matrix
model considered in this paper is widely used in the fields of multidimensional
signal processing and radio communications. Using the invariance of the
probability law of $X_n$ by multiplication by a constant unitary matrix, $D_n$
can be straightforwardly replaced with a nonnegative Hermitian matrix $R_n$.
In the model $X_n R_n^{1/2} + P_n$ where $R_n^{1/2}$ is any square root of
$R_n$, matrix $P_n$ often represents $n$ snapshots of a discrete time radio
signal sent by $r$ sources and received by an array of $N$ antennas, while $X_n
R_n^{1/2}$ is a temporally correlated and spatially independent ``noise''
(spatially correlated and temporally independent noises can be considered as
well). In this framework, the results of this paper can be used for detecting
the signal sources, estimating their powers, or determining their directions.
These subjects are explored in the applicative paper
\cite{vino-etal-ieeesp-sub12}. \\ 

The remainder of the article is organized as follows. The assumptions and the
main results are provided in Section \ref{pb-results}. The general approach as
well as the basic mathematical tools needed for the proofs are provided in
Secion \ref{basic}. These proofs are given in Sections \ref{1st-order} and
\ref{sec:2ord}, which concern respectively the first order (convergence) and
the second order (fluctuations) behavior of the outliers.

\section{Problem description and main results} 
\label{pb-results} 

Given a sequence of integers $N=N(n)$, $n=1, 2, \ldots$, we consider the 
sequence of $N\times n$ matrices $\Sigma_n = n^{-1/2} Y_n + P_n = 
n^{-1/2} X_n D_n^{1/2} + P_n$ with the following assumptions: 

\begin{assumption}
\label{regime} 
The ratio $c_n = N(n) / n$ converges to a positive constant $c$ as 
$n\to\infty$. 
\end{assumption} 

\begin{assumption}
\label{X:gauss}
The matrix $X_n = [X_{ij}]_{i,j=1}^{N,n}$ is a $N\times n$ random matrix whose 
coefficients $X_{ij}$ are iid complex random variables such that 
$\Re(X_{11})$ and $\Im(X_{11})$ are independent, each with probability
distribution ${\mathcal N}(0,1/2)$. 
\end{assumption} 

\begin{assumption}
\label{assump:D} 
The sequence of $n\times n$ deterministic diagonal nonnegative matrices 
$D_n = \diag( d_j^n)_{j=1}^n$ satisfies the following:
\begin{enumerate}
\item 
\label{msl-D} 
The probability measure $\nu_n = n^{-1} \sum_{j=1}^n 
\bs\delta_{d_j^n}$ converges weakly to a probability measure $\nu$ 
with compact support. 
\item 
\label{D:noeigenvalues} 
The distances ${\bs d}( d_j^n, \support(\nu) )$ from $ d_j^n$ 
to $\support(\nu)$ satisfy 
\[
\max_{j\in \{1,\ldots,n\}}
 {\bs d}\left( d_j^n, \support(\nu) \right)
\xrightarrow[n\to\infty]{}  0 . 
\] 
\end{enumerate} 
\end{assumption} 

The asymptotic behavior of the spectral measure of $n^{-1} Y_n Y_n^*$ under
these assumptions has been thoroughly studied in the literature. 
Before pursuing, we recall the main results which describe this behavior.
These results are built around the Stieltjes Transform, defined, for a positive finite 
measure $\mu$ over the Borel sets of $\R$, as 
\begin{equation}
\label{ST} 
m(z) = \int_\R \frac{1}{t-z} \mu(dt) 
\end{equation} 
analytic on $\C - \support(\mu)$. It is straightforward to check that 
$\Im m(z) \geq 0$ when $z \in \C_+ = \{ z \, : \, \Im(z) > 0 \}$, and 
$\sup_{y > 0} | y m(\imath y) | < \infty$. Conversely,
any analytic function $m(z)$ on $\C_+$ that has these two properties 
admits the integral representation \eqref{ST} where $\mu$ is a positive finite 
measure. 
Furthermore, for any continuous real function $\varphi$ with compact support 
in $\R$, 
\begin{equation}
\label{perron} 
\int \varphi(t) \, \mu(dt) = \frac 1\pi \lim_{y\downarrow 0} 
\int \varphi(x) \Im m(x+ \imath y) \, dx 
\end{equation} 
which implies that the measure $\mu$ is uniquely defined by its Stieltjes
Transform. Finally, if $\Im(z m(z)) \geq 0$ when $z \in \C_+$, then 
$\mu((-\infty, 0)) = 0$ \cite{krein-nudel}. \\
These facts can be generalized to Hermitian matrix-valued 
nonnegative finite measures \cite{bolot97,Gesztesy-herglotz}. 
Let $m(z)$ be a $\C^{r\times r}$-valued analytic function on $z \in \C_+$.
Letting $\Im X = ( X - X^* ) / (2 \imath)$, assume that $\Im m(z) \geq 0$ 
and $\Im (z m(z)) \geq 0$ in the order of the Hermitian matrices for any
$z\in \C_+$, and that $\sup_{y > 0} \| y m(\imath y) \| < \infty$. 
Then $m(z)$ admits the representation \eqref{ST} where $\mu$ is now a 
$r\times r$ matrix-valued nonnegative finite measure such that 
$\mu((-\infty, 0)) = 0$. One can also check 
that $\mu([0,\infty)) = - \lim_{y\to\infty} \imath y \, m(-\imath y)$. \\

The first part of the following theorem has been shown in 
\cite{MarPas67,SilBai95}, and the second part in \cite{BaiSil98}: 

\begin{theorem} 
\label{lsm}
Under Assumptions \ref{regime}, \ref{X:gauss} and \ref{assump:D}, the 
following hold true: 
\begin{enumerate}
\item\label{st-lsm} 
For any $z \in \C_+$, the equation 
\begin{equation} 
\label{m=f(m)} 
\mfat = \left( -z + \int \frac{t}{1 + c \mfat t} \nu(dt) \right)^{-1} 
\end{equation}
admits a unique solution $\mfat \in \C_+$. The function $\mfat = \mfat(z)$ so 
defined on $\C_+$ is the Stieltjes Transform 
of a probability measure $\mu$ whose support is a compact set of $\R_+$. \\
Let $(\lambda_{i}^n)_{i=1,\ldots,N}$ be the eigenvalues of $n^{-1} Y_n Y_n^*$,
and let $\theta_n = N^{-1} \sum_{i=1}^N \bs\delta_{\lambda_{i}^n}$ be the 
spectral measure of this matrix. Then for every bounded and continuous real 
function $f$, 
\begin{equation}
\label{ascvg} 
\int f(t) \theta_n(dt) \ \toaslong \ 
\int f(t) \mu(dt) . 
\end{equation} 
\item\label{noeig-sil} 
For any interval $[x_1, x_2] \subset \R - \support(\mu)$, 
\[
\sharp \{ i \, : \, \lambda_{i}^n \in [x_1, x_2] \} 
= 0 \ \text{with probability} \ 1 \ \text{for all large} \ n. 
\]
\end{enumerate} 
\end{theorem} 

We now consider the additive deformation $P_n$: 
\begin{assumption}
\label{ass:P}
The deterministic $N \times n$ matrices $P_n$ have a fixed rank equal to $r$. 
Moreover, $\pmax = \sup_n \| P_n \| < \infty$. 
\end{assumption} 

In order for some of the eigenvalues of $\Sigma_n\Sigma_n^*$ to converge to 
values outside $\support(\mu)$, an extra assumption involving in some sense
the interaction between $P_n$ and $D_n$ is needed. 
Let $P_n = U_n^{\text{GS}} (R_n^{\text{GS}})^*$ be the Gram-Schmidt factorization of 
$P_n$ where $U_n^{\text{GS}}$ is an isometry $N \times r$ matrix and where 
$(R_n^{\text{GS}})^*$ is an upper triangular matrix in row echelon form whose first nonzero coefficient of each row is positive. The factorization so defined  is then  unique. Define the
$r \times r$ Hermitian nonnegative matrix-valued measure 
$\Lambda_n^{\text{GS}}$ as 
\[
\Lambda_n^{\text{GS}}(dt) = (R_n^{\text{GS}})^* 
\begin{bmatrix} 
\bs\delta_{ d_1^n}(dt) \\ & \ddots \\ & & \bs\delta_{ d_n^n}(dt) 
\end{bmatrix}
R_n^{\text{GS}} .  
\] 
Assumption \ref{assump:D} shows that $\dmax = \sup_n \|  D_n \| < \infty$.
Moreover, it is clear that the support of $\Lambda_n^{\text{GS}}$ is included 
in $[0, \dmax ]$ and that $\Lambda_n^{\text{GS}}([0,\dmax]) \leq \pmax^2 I_r$. 
Since the sequence $\Lambda_n^{\text{GS}}([0,\dmax])$ is bounded in norm, for 
every sequence of integers increasing to infinity, there exists a subsequence 
$n_k$ and a nonnegative finite 
measure $\Lambda_*$ such that $\int f(t) \Lambda_{n_k}^{\text{GS}}(dt) \to 
\int f(t) \Lambda_*(dt)$ for every function $f \in {\cal C}([0, \dmax ])$, 
with ${\cal C}([0, \dmax ])$ being the set of continuous functions on 
$[0, \dmax ]$. This fact is a straightforward extension of its analogue for 
scalar measures.  

\begin{assumption}
\label{ass:cvg}
Any two accumulation points $\Lambda_1$ and $\Lambda_2$ of the sequences 
$\Lambda_n^{\operatorname{GS}}$ satisfy $\Lambda_1(dx) = W \Lambda_2(dx) W^*$ where 
$W$ is a $r\times r$ unitary matrix. 
\end{assumption}

This assumption on the interaction between $P_n$ and $D_n$ appears to be the
least restrictive assumption ensuring the convergence of the outliers 
to fixed values outside $\support(\mu)$ as $n\to\infty$. 
If we consider some other factorization $P_n = U_n R_n^*$ of $P_n$ where $U_n$
is an isometry matrix with size $N \times r$, and if we associate to the
$R_n$ the sequence of $r \times r$ Hermitian nonnegative matrix-valued 
measures $\Lambda_n$ defined as 
\begin{equation} 
\label{def-Lambda} 
\Lambda_n(dt) = R_n^* 
\begin{bmatrix} 
\bs\delta_{ d_1^n}(dt) \\ & \ddots \\ & & \bs\delta_{ d_n^n}(dt) 
\end{bmatrix}
R_n  
\end{equation}
then it is clear that $\Lambda_n(dt) = W_n \Lambda_n^{\text{GS}}(dt) W_n^*$
for some $r \times r$ unitary matrix $W_n$. By the compactness
of the unitary group, Assumption~\ref{ass:cvg} is satisfied for 
$\Lambda_n^{\text{GS}}$ if and only if it is satisfied for $\Lambda_n$. 
The main consequence of this assumption is that for any function 
$f \in {\cal C}([0, \dmax ])$, the eigenvalues of the matrix 
$\int f(t) \Lambda_n(dt)$ arranged in some given order will converge. \\ 

An example taken from the fields of signal processing and wireless 
communications might help to have a better understanding the applicability 
of Assumption~\ref{ass:cvg}. 
In these fields, the matrix $P_n$ often represents a multidimensional radio 
signal received by an array of $N$ antennas. 
Frequently this matrix can be factored as $P_n = n^{-1/2} U_n A_n S_n^*$ where 
$U_n$ is a deterministic $N \times r$ isometry matrix, $A_n$ is a deterministic
$r \times r$ matrix such that $A_n A_n^*$ converges to a matrix $M$ as 
$n \to\infty$ (one often assumes $A_n A_n^*=M$ for each $n$), and $S_n = [S_{ij}]_{i,j=1}^{n,r}$ is a $n \times r$ random 
matrix independent of $X_n$ with iid elements satisfying 
$\E S_{1,1} = 0$ and $\E|S_{1,1}|^2 = 1$ (in the wireless communications 
terminology, $U_n A_n$
is the so called MIMO channel matrix and $S_n$ is the so called signal matrix, 
see \cite{tse-vis-livre-05}). Taking $R_n = n^{-1/2} S_n A_n^*$ in 
\eqref{def-Lambda} and applying the law of large numbers, one can see that
for any $f \in {\cal C}([0, \dmax ])$, the integral $\int f(t) \Lambda_n(dt)$ 
converges to $\int f(t) \Lambda_*(dt)$ with $\Lambda_*(dt) = \nu(dt) \times M$.
Clearly, the accumulation points of the measures obtained from any other
sequence of factorizations of $P_n$ are of the form $\nu(dt) \times W M W^*$ 
where $W$ is an $r\times r$ unitary matrix. \\ 

It is shown in \cite{sil-choi95} that the limiting spectral measure $\mu$ has a 
continuous density on $\R^* = \R - \{ 0 \}$ (see Prop. \ref{lim-m(z)} below). 
Our first order result addresses the problem of the presence 
of isolated eigenvalues of $\Sigma_n \Sigma_n^*$ in any compact interval 
outside the support of this density. Of prime importance will be the 
$r\times r$ matrix functions 
\[
H_*(z) = \int \frac{\mfat(z)}{1 + c \mfat(z) t} \Lambda_*(dt) 
\]
where $\Lambda_*$ is an accumulation point of a sequence $\Lambda_n$. 
Since $|1 + c \mfat(z) t | = | z(1 + c \mfat(z) t) | / |z| \geq 
| \Im( z(1 + c \mfat(z) t) ) | / |z| \geq \Im(z) / |z|$ on $\C_+$, the function 
$H_*(z)$ is analytic on $\C_+$. It is further easy to show that 
$\Im(H_*(z)) \geq 0$ and $\Im ( z H_*(z) ) \geq 0$ on $\C_+$,
and $\sup_{y > 0} \| y H_*(\imath y) \| < \infty$. Hence $H_*(z)$ is the 
Stieltjes Transform of a matrix-valued nonnegative finite measure carried
by $[0, \infty)$. Note also that, under Assumption \ref{ass:cvg}, the 
eigenvalues of $H_*(z)$ remain unchanged if $\Lambda_*$ is replaced by
another accumulation point. 

The support of $\mu$ may consist in several connected components 
corresponding to as many ``bulks'' of eigenvalues. Our first theorem 
specifies the locations of the outliers between any two bulks and on the right of the 
last bulk. It also shows that there are no outliers on the left of the first bulk: 

\begin{theorem}
\label{1st-ord} 
Let Assumptions \ref{regime}, \ref{X:gauss} and \ref{assump:D} hold true. 
Denote by $(\hat\lambda_{i}^n)_{i=1,\ldots,N}$ the eigenvalues of
$\Sigma_n\Sigma_n^*$. Let $(a,b)$ be any connected component of  
$\support(\mu)^c = \R - \support(\mu)$. 
Then the following facts hold true: 

\begin{enumerate}
\item\label{loc-spikes} 
Let $(P_n)$ be a sequence satisfying Assumptions \ref{ass:P} and 
\ref{ass:cvg}. Given an accumulation point $\Lambda_*$ of a sequence 
$\Lambda_n$, let $H_*(z)= \int \mfat(z)(1 + c \mfat(z) t)^{-1} \Lambda_*(dt)$.  
Then $H_*(z)$ can be analytically extended to $(a,b)$ where its values are  
Hermitian matrices, and the extension is increasing in the order of Hermitian 
matrices on $(a,b)$. The function ${\cal D}(x) = \det( H_*(x) + I_r )$ has at
most $r$ zeros on $(a,b)$.
Let $\rho_1, \ldots, \rho_k$, $k\leq r$ be these zeros counting multiplicities.
Let $\cal I$ be any compact interval in $(a,b)$ such that 
$\{ \rho_1, \ldots, \rho_k \} \cap \partial\cal I = \emptyset$. Then 
\[
\sharp\{ i \, : \, \hat\lambda_{i}^n \in \cal I \} =  
\sharp\{ i \, : \, \rho_{i} \in \cal I \}   
\ \text{with probability} \ 1 \ \text{for all large} \ n. 
\]

\item
\label{no-spikes} 
Let $A = \inf\left( \support(\mu) - \{ 0 \}\right)$. Then for any positive 
$A' < A$ (assuming it exists) and for any sequence of matrices $(P_n)$ 
satisfying Assumption \ref{ass:P}, 
\[
\sharp\{ i \, : \, \hat\lambda_{i}^n \in (0, A'] \} = 0 
\ \text{with probability} \ 1 \ \text{for all large} \ n. 
\]

\end{enumerate}
\end{theorem} 

Given any sequence of positive real numbers $\rho_1 \leq \cdots \leq \rho_r$
lying in a connected component of $\support(\mu)^c$ after the first bulk, it 
would be interesting to see whether there exists a sequence of matrices $P_n$ 
that produces outliers converging to these $\rho_k$. The following theorem 
answers this question positively: 
\begin{theorem}
\label{partout}
Let Assumptions \ref{regime}, \ref{X:gauss} and \ref{assump:D} hold true. 
Let $\rho_1 \leq \cdots \leq \rho_r$ be a sequence of positive real numbers
lying in a connected component $(a,b)$ of $\support(\mu)^c$, and such that 
$a > A$. Then there exists a 
sequence of matrices $P_n$ satisfying Assumptions \ref{ass:P} and 
\ref{ass:cvg} such that for any compact interval $\cal I \subset (a,b)$ with 
$\{ \rho_1, \ldots, \rho_r \} \cap \partial \cal I = \emptyset$, 
\[
\sharp\{ i \, : \, \hat\lambda_{i}^n \in \cal I \} =  
\sharp\{ i \, : \, \rho_{i} \in \cal I \}   
\ \text{with probability} \ 1 \ \text{for all large} \ n. 
\]
\end{theorem} 

It would be interesting to complete the results of these theorems by specifying
the indices of the outliers $\hat\lambda_i^n$ that appear between the bulks. 
This demanding analysis might be done by following the 
ideas of \cite{cdf09} or \cite{val-lou-mes-it12} relative to the so called 
separation of the eigenvalues of $\Sigma_n\Sigma_n^*$. Another approach 
dealing with the same kind of problem is developed in 
\cite{bai-sil-separation}. \\ 

A case of practical importance at least in the domain of signal processing 
is described by the following assumption:

\begin{assumption}
\label{ass:sp}
The accumulation points $\Lambda_*$ are of the form 
$\nu(dt) \times W {\bs \Omega} W^*$ where 
\[
{\bs \Omega} = \begin{bmatrix} \omega_1^2 I_{j_1}  \\ & \ddots \\ & & 
\omega_t^2 I_{j_t}
\end{bmatrix} > 0 , \quad \omega_1^2 > \cdots > \omega_t^2, 
\quad j_1 + \cdots + j_t = r  
\]
and where $W$ is a unitary matrix. 
\end{assumption} 

Because of the specific structure of $S_n$ in the factorization
$P_n=n^{-1/2}U_nA_nS_n^*$, the MIMO wireless communication model described
above satisfies this assumption, the $\omega_i^2$ often referring to the powers
of the radio sources transmitting their signals to an array of antennas. \\
Another case where Assumption \ref{ass:sp} is satisfied is the case where $P_n$
is a random matrix independent of $X_n$, where its probability distribution is
invariant by right multiplication with a constant unitary matrix, and where the
$r$ non zero singular values of $P_n$ converge almost surely towards constant 
values.  \\
When this assumption is satisfied, we obtain the following corollary of Theorem~\ref{1st-ord} which exhibits some sort of phase transition analogous to the so-called BBP phase transition~\cite{bbp05}: 

\begin{corollary}
\label{1stord-sp}
Assume the setting of Theorem \ref{1st-ord}-\eqref{loc-spikes}, and let 
Assumption \ref{ass:sp} hold true. Then the function 
$g(x) = \mfat(x) \left( c x \mfat(x) - 1 + c \right)$ is decreasing on $(a,b)$. 
Depending on the value of $\omega_\ell^2$, $\ell=1,\ldots, t$, the equation 
$\omega_\ell^2 g(x) = 1$ has either zero or one solution in $(a,b)$. Denote 
by $\rho_1, \ldots, \rho_k$, $k\leq r$ these solutions counting 
multiplicities. Then the conclusion of Theorem \ref{1st-ord}-\eqref{loc-spikes}
hold true for these $\rho_i$. 
\end{corollary} 


We now turn to the second order result, which will be stated in the 
simple and practical framework of Assumption \ref{ass:sp}. Actually, 
a stronger assumption is needed: 
\begin{assumption}
\label{fast-cvg} 
The following facts hold true: 
\begin{gather*} 
\sup_n \sqrt{n} | c_n - c | < \infty , \\ 
\limsup_n \sqrt{n} \,  
\Bigl | \int \frac{1}{t-x} \nu_n(dt) \ - \ 
 \int \frac{1}{t-x} \nu(dt) \Bigr | < \infty 
\ \text{for all} \ x \in \R - \support(\nu) . 
\end{gather*} 
Moreover, there exists a sequence of factorizations of $P_n$ such that the
measures $\Lambda_n$ associated with these factorizations by 
\eqref{def-Lambda} converge to $\nu(dt) \times {\bs\Omega}$ and such that
\[
\limsup_n \sqrt{n} \, 
\Bigl \| \int \frac{1}{t-x} \Lambda_n(dt) \ - \ 
 \int \frac{1}{t-x} \nu(dt)  \times  
{\bs \Omega} \Bigr\| < \infty  \ \text{for all} \ x \in \R - \support(\nu) 
\] 
\end{assumption}

Note that one could have considered the above superior limits to be zero, which would simplify the statement of Theorem~\ref{th-2ord} below. However, in practice this is usually too strong a requirement, see {\it e.g.} the wireless 
communications model discussed after Assumption~\ref{ass:cvg} for which the fluctuations of $\Lambda_n$ are of order $n^{-1/2}$. On the opposite, slower fluctuations of $\Lambda_n$ would result in a much more intricate result for Theorem~\ref{th-2ord}, which we do not consider here.

Before stating the second order result, a refinement of the results of Theorem
\ref{lsm}--\eqref{st-lsm} is needed: 
\begin{proposition}[\cite{SilBai95,hachem-loubaton-najim07, HLN08Ieee}] 
\label{equiv-deter}
Assume that $ D_n$ is a $n \times n$ diagonal nonnegative matrix. Then, for 
any $n$, the equation 
\[
m_n = \left[ -z \left( 1 + \frac 1n \tr  D_n \widetilde T_n 
\right) \right]^{-1} 
\quad \text{where} \quad 
\widetilde T_n = 
\left[ -z\left( I_n + c_n m_n  D_n \right) \right]^{-1}  
\] 
admits a unique solution $m_n \in \C_+$ for any $z \in \C_+$. The function 
$m_n = m_n(z)$ so defined on $\C_+$ is the Stieltjes Transform 
of a probability measure $\mu_n$ whose support is a compact set of $\R_+$.  
Moreover, the $n\times n$ diagonal matrix-valued function 
$\widetilde T_n(z) = [ -z( I_n + c_n m_n(z) D_n )]^{-1}$ 
is analytic on $\C_+$ and $n^{-1} \tr \widetilde T_n(z)$
coincides with the Stieltjes Transform of 
$c_n \mu_n + ( 1 - c_n ) \bs\delta_0$. \\
Let Assumption \ref{X:gauss} hold true, and assume that $\sup_n 
\|  D_n \| < \infty$, and $0 < \liminf c_n \leq \limsup c_n < \infty$. Then 
the resolvents $Q_n(z) = (n^{-1} Y_n Y_n^* - z I_N)^{-1}$ and 
$\widetilde Q_n(z) = (n^{-1} Y_n^* Y_n - z I_n)^{-1}$ satisfy 
\begin{equation} 
\label{cvg-resolv} 
\frac 1N \tr \left( Q_n(z) - m_n(z) I_N \right) \toaslong 0 
\quad \text{and} \quad 
\frac 1n \tr \left( \widetilde Q_n(z) - \widetilde T_n(z) \right) 
\toaslong 0 
\end{equation} 
for any $z \in \C_+$. When in addition Assumptions \ref{regime} and 
\ref{assump:D} hold true, $m_n(z)$ converges to ${\bf m}(z)$ provided in the 
statement of Theorem \ref{lsm} uniformly on the compact subsets of $\C_+$. 
\end{proposition} 
The function 
$m_n(z) = \left( -z + \int t (1 + c_n m_n(z) t)^{-1} \nu_n(dt) \right)^{-1}$ 
is a finite $n$ approximation of $\mfat(z)$. 
Notice that since $N^{-1} \tr Q_n(z)$ is the Stieltjes Transform of the 
spectral measure $\theta_n$ of $n^{-1} Y_n Y_n^*$, Convergence \eqref{ascvg} 
stems from \eqref{cvg-resolv}. \\
We shall also need a finite $n$ approximation of $H_*(z)$ defined as 
\[
H_n(z) = \int \frac{m_n(z)}{1+c_n m_n(z) t} \Lambda_n(dt) . 
\] 
With these definitions, we have the following preliminary proposition: 
\begin{proposition}
\label{var-biais}
Let Assumptions \ref{regime}, \ref{assump:D}-\ref{fast-cvg} hold true.  
Let $g$ be the function defined in the statement of Corollary \ref{1stord-sp} 
and let $B_\mu = \sup(\support(\mu))$. Assume that the equation 
$\omega_1^2 g(x) = 1$ has a solution in $(B_\mu, \infty)$, and denote 
$\rho_1 > \cdots > \rho_p$ the existing solutions (with respective 
multiplicities $j_1, \ldots, j_p$) of the equations 
$\omega_\ell^2 g(x) = 1$ in $(B_\mu, \infty)$. Then the following facts 
hold true: 
\begin{itemize}
\item 
$\displaystyle{{\bs \Delta}(\rho_i) = 
1 - c \int \left(\frac{\mfat(\rho_i) t}{1+c\mfat(\rho_i) t}\right)^2 \nu(dt)}$
is positive for every $i=1,\ldots, p$. 
\item 
Denoting by $H_{1,n}(z), \ldots, H_{p,n}(z)$ the first $p$ upper left diagonal 
blocks of $H_n(z)$, where $H_{i,n}(z) \in \C^{j_i\times j_i}$, 
$\limsup_n \| \sqrt{n}(H_{i,n}(\rho_i)+I_{j_i}) \| < \infty$ for every
$i = 1,\ldots, p$. 
\end{itemize} 
\end{proposition} 

We recall that a GUE matrix (\emph{i.e.}, a matrix taken from the Gaussian 
Unitary Ensemble) is a random Hermitian matrix $G$ such that 
$G_{ii} \sim {\cal N}(0,1)$, $\Re(G_{ij}) \sim {\cal N}(0,1/2)$ and 
$\Im(G_{ij}) \sim {\cal N}(0,1/2)$ for $i < j$, and such that all these
random variables are independent. Our second order result is provided by
the following theorem: 

\begin{theorem}
\label{th-2ord} 
Let Assumptions \ref{regime}-\ref{fast-cvg} hold true. Keeping the notations
of Proposition \ref{var-biais}, let 
\[
M_i^n = \sqrt{n} \left( 
\begin{bmatrix} 
\hat \lambda^n_{j_1 + \cdots + j_{i-1} + 1} \\
\vdots \\ 
\hat \lambda^n_{j_1 + \cdots + j_{i}} 
\end{bmatrix} 
- 
\rho_i \begin{bmatrix} 1 \\ \vdots \\ 1 \end{bmatrix} 
\right)
\]
where $j_0 = 0$ and where the eigenvalues $\hat \lambda^n_i$ of 
$\Sigma_n \Sigma_n^*$ are arranged in decreasing order. 
Let $G_1, \ldots, G_p$ be independent GUE matrices such that $G_i$ is a
$j_i \times j_i$ matrix. Then, for any bounded continuous 
$f\colon\R^{j_1+\ldots+j_p}\to \R$,
\[
\EE\left[ f\left( M_1^n,\ldots,M_p^n\right) \right] - \EE \left[ f\left(\Xi_1^n,\ldots,\Xi_p^n\right) \right] \underset{n\to\infty}{\longrightarrow} 0
\]
where $\Xi_i^n\in \R^{j_i}$ is the random vector 
of the decreasingly ordered eigenvalues of the matrix
\begin{align*}
	 \frac{1}{\omega_i^2 g(\rho_i)'} \left({\bs\alpha}_i G_i + \sqrt{n}\left(H_{i,n}(\rho_i)+I_{j_i} \right)\right) ,
\end{align*}
 where
\[ 
{\bs\alpha}_i^2 = \frac{\mfat^2(\rho_i)}{ {\bs\Delta}(\rho_i) } 
\left[ \int \frac{t^2+ 2\omega_i^2 t}{(1+c\mfat(\rho_i)t)^2} \nu(dt) 
+ c \left( \int \frac{\omega_i^2 \mfat(\rho_i)t}
{(1 + c{\bf m}(\rho_i) t)^2}\nu(dt) \right)^2 \right] . 
\] 
\end{theorem}

Some remarks can be useful at this stage. The first remark concerns Assumption 
\ref{fast-cvg}, which is in some sense analogous to \cite[Hypothesis           
3.1]{bgm-fluct10}. This assumption is mainly needed to show that the $\sqrt{n} 
\| H_{i,n}(\rho_i) + I_{j_i} \|$ are bounded, guaranteeing the tightness of    
the vectors $M_{i}^n$.                                                         
Assuming that $\Lambda_n$ and $\Lambda'_n$ both satisfy the third item
of Assumption \ref{fast-cvg}, denoting respectively by $H_{i,n}(\rho_i)$ and 
$H'_{i,n}(\rho_i)$ the matrices associated to these measures as in the 
statement of Theorem \ref{th-2ord}, it is possible to show that 
$\sqrt{n} ( \sigma(H_{i,n}(\rho_i) - \sigma(H'_{i,n}(\rho_i) ) \to 0$ as
$n\to\infty$. Thus the results of this theorem do not depend on the 
particular measure $\Lambda_n$ satisfying Assumption \ref{fast-cvg}.           
Finally, we note that Assumption \ref{fast-cvg} can be lightened at the        
expense of replacing the limit values $\rho_i$ with certain finite $n$         
approximations of the outliers, as is done in the applicative paper            
\cite{vino-etal-ieeesp-sub12}. \\

The second remark pertains to the Gaussian assumption on the elements of
$X_n$. We shall see below that the results of Theorems~\ref{1st-ord}--\ref{th-2ord} are intimately related to the first and second 
order behaviors of bilinear forms of the type 
$u_n^* Q_n(x) v_n$, $\tilde u_n^* \widetilde Q_n(x) \tilde v_n$, and 
$n^{-1/2} u_n^* Y_n \widetilde Q_n(x) \tilde v_n$ 
where $u_n$, $v_n$, 
$\tilde u_n$ and $\tilde v_n$ are deterministic vectors of bounded norm and 
of appropriate dimensions, and where $x$ is a real number lying outside the 
support of $\mu$. 
In fact, it is possible to generalize Theorems \ref{1st-ord} and \ref{partout}
to the case where the elements of $X_n$ are not necessarily Gaussian. This can 
be made possible by using the technique of \cite{hlnv-bilinear-ihp11} to 
analyze the first order behavior of these bilinear forms. 
On the other hand, the Gaussian assumption plays a central role in Theorem 
\ref{th-2ord}. Indeed, the proof of this theorem is based on the fact that 
these bilinear forms asymptotically fluctuate like Gaussian random 
variables when centered and scaled by $\sqrt{n}$. 
Take $u_n = e_{1,N}$ and $\tilde v_n = e_{1,n}$ where $e_{k,m}$ is the 
$k^{\text{th}}$ canonical vector of $\R^m$. We show below (see 
Proposition \ref{equiv-deter} and Lemmas \ref{lm-var} and \ref{mean-fq}) 
that the elements $\tilde q_{ij}^n$ of the
resolvent $\widetilde Q_n(x)$ are close for large $n$ to the elements 
$\tilde t_{ij}^n$ of the deterministic matrix $\widetilde T_n(x)$. 
We therefore write informally 
\[
e_{1,N}^* Y_n \widetilde Q(x) e_{1,n} = 
\sum_{j=1}^n ( d_{j}^n)^{1/2} \tilde q_{j1}^n X_{1j} 
\approx 
( d_{1}^n)^{1/2} \tilde t_{11}^n X_{11} 
+ 
\sum_{j=2}^n ( d_{j}^n)^{1/2} \tilde q_{j1}^n X_{1j} .
\]
It can be shown furthermore that 
$\tilde t_{11}^n = {\mathcal O}(1)$ for large $n$ and that the sum 
$\sum_{j=2}^n$ is tight. Hence, $e_{1,N}^* Y_n \widetilde Q(x) e_{1,n}$ is
tight. However, when $X_{11}$ is not Gaussian, we infer that 
$e_{1,N}^* Y_n \widetilde Q(x) e_{1,n}$ does not converge in general towards 
a Gaussian random variable. In this case, if we choose 
$P_n = \omega^2 e_{1,N} e_{1,n}^*$ (see Section \ref{sec:2ord}), Theorem 
\ref{th-2ord} no longer holds. 
Yet, we conjecture that an analogue of this theorem can be recovered when 
$e_{1,N}$ and $e_{1,n}$ are replaced with delocalized vectors, following the 
terminology of \cite{cdf-ihp12}. In a word, the elements of these vectors are
``spread enough'' so that the Gaussian fluctuations are recovered.

\subsection*{A word about the notations}
In the remainder of the paper, we shall often drop the subscript or the 
superscript $n$ when there is no ambiguity. A constant bound that may change 
from an inequality to another but which is independent of $n$ will always be 
denoted $K$. Element $(i,j)$ of matrix $M$ is denoted $M_{ij}$ or $[M]_{ij}$. 
Element $i$ of vector $x$ is denoted $[x]_i$. 
Convergences in the almost sure sense, in probability and in distribution 
will be respectively denoted $\overset{{\text{a.s.}}}{\longrightarrow}$, 
$\overset{\mathcal P}{\longrightarrow}$, and 
$\overset{\mathcal D}{\longrightarrow}$.

\section{Preliminaries and useful results} 
\label{basic}

We start this section by providing the main ideas of the proofs of 
Theorems \ref{1st-ord} and \ref{partout}.

\subsection{Proof principles of the first order results} 
\label{idee-preuve}

The proof of Theorem \ref{1st-ord}-\eqref{loc-spikes}, to begin with, is based 
on the idea of \cite{ben-rao-published, ben-rao-11}. We start with a purely 
algebraic result. 
Let $P_n = U_n R_n^*$ be a factorization of $P_n$ where $U_n$ is a $N \times r$
isometry matrix. 
Assume that $x > 0$ is not an eigenvalue of $n^{-1} Y_nY_n^*$. Then $x$ is an 
eigenvalue of $\Sigma_n\Sigma_n^*$ if and only if $\det\widehat S_n(x) = 0$ 
where $\widehat S_n(x)$ is the $2r \times 2r$ matrix 
\[
\widehat S_n(x) = 
\begin{bmatrix}
\sqrt{x} U_n^* Q_n(x) U_n & I_r + n^{-1/2} U_n^* Y_n \widetilde Q_n(x) R_n \\
I_r + n^{-1/2} R_n^* \widetilde Q_n(x) Y_n^* U_n & 
\sqrt{x} R_n^* \widetilde Q_n(x) R_n 
\end{bmatrix}  
\]
(for details, see the derivations in \cite{ben-rao-11} or in 
\cite[Section 3]{hlmnv-subspace11}). The idea is now the following. 
Set $x$ in $\support(\mu)^c$. Using an integration by parts 
formula for functionals of Gaussian vectors and the Poincar\'e-Nash inequality 
\cite{pas-livre}, we show that when $n$ is large,  
\begin{gather*} 
U_n^* Q_n(x) U_n \approx m_n(x) I_r, \ 
R_n^* \widetilde Q_n(x) R_n \approx R_n^* \widetilde T_n(x) R_n, \\ 
\text{and} \ 
n^{-1/2} R_n^* \widetilde Q_n(x) Y_n^* U_n \approx 0  
\end{gather*} 
by controlling the moments of the elements of the left hand members. To be 
able to do these controls, we make use of a certain regularizing function
which controls the escape of the eigenvalues of $n^{-1} Y_n Y_n^*$ out of
$\support(\mu)$. Thanks to these results, $\widehat S_n(x)$ is close for 
large $n$ to 
\[
S_n(x) = 
\begin{bmatrix}
\sqrt{x} m_n(x) I_r & I_r \\
I_r & \sqrt{x} R_n^* \widetilde T_n(x) R_n 
\end{bmatrix} .
\]
Hence, we expect the eigenvalues of $\Sigma_n\Sigma_n^*$ in the interval
$\cal I$, when they exist, to be close for large $n$ to the zeros in $\cal I$ 
of the function 
\begin{align*} 
\det S_n(x) &= 
\det\left( x m_n(x) R_n^* \widetilde T_n(x) R_n - I_r \right) \\
&= (-1)^r \det\left( m_n(x) R_n^* (I_n + c_n m_n(x) D_n)^{-1} R_n^* 
                                                          + I_r \right) \\
&= (-1)^r \det\left( H_n(x) + I_r \right) 
\end{align*} 
which are close to the zeros of ${\cal D}(x) = \det ( H_*(x) + I_r)$. By 
Assumption  \ref{ass:cvg}, these zeros are independent of the choice of the 
accumulation point $\Lambda_*$.  \\

To prove Theorems \ref{1st-ord}-\eqref{no-spikes} and \ref{partout}, we make
use of the results of \cite{sil-choi95} and \cite{mes-it08,mes-sp08} relative 
to the properties of $\mu$ and to those of the restriction of ${\bf m}(z)$ to 
$\R-\support(\mu)$. 
The main idea is to show that 
\begin{itemize}
\item ${\bf m}(x)(1+c {\bf m}(x)t)^{-1} > 0$ for all 
$x \in \support(\mu)^c \cap (-\infty, A)$ 
(these $x$ lie at the left of the first bulk) and for all 
$t \in \support(\nu)$. 
\item For any component $(a,b) \subset \support(\mu)^c$ such that $a > A$
(\emph{i.e.}, lying between two bulks or at the right of the last bulk), 
there exists a Borel set $E \subset \R_+$ such that $\nu(E) > 0$ and  
\[
q(x) = \int_E \frac{\mfat(x)}{1+c\mfat(x)t} \nu(dt) < 0
\]
for all $x \in (a,b)$.     
\end{itemize}
Thanks to the first result, for any $x$ lying if possible between zero and
the left edge of the first bulk, ${\cal D}(x) > 0$, hence $\Sigma_n\Sigma_n^*$
has asymptotically no outlier at the left of the first bulk. \\
Coming to Theorem \ref{partout}, let $E$ be a set associated to $(a,b)$ by 
the result above. We build a sequence of matrices $P_n$ of rank $r$, and such 
that the associated $\Lambda_n$ have an accumulation point of the form 
$\Lambda_*(dt) = \1_E(t) \, \bs\Omega \, \nu(dt)$ where we choose 
$\bs\Omega = \diag( -q(\rho_1)^{-1}, \ldots, -q(\rho_r)^{-1} )$.  
Theorem \ref{1st-ord}-\eqref{loc-spikes} shows that the function 
$H_*(x) = q(x) \bs \Omega$ associated with this $\Lambda_*$ is 
increasing on $(a,b)$. As a result, $H_*(x) + I_r$ becomes singular 
precisely at the points $\rho_1, \ldots, \rho_r$. 

\subsection{Sketch of the proof of the second order result}

The fluctuations of the outliers will be deduced from the fluctuations of the 
elements of the matrices $\widehat S_n(\rho_i)$ introduced above. The proof of Theorem 
\ref{th-2ord} can be divided into two main steps. The first step 
(Lemma \ref{le:3p}) consists in establishing a Central Limit Theorem on the 
$3p$--uple of random matrices 
\begin{eqnarray*} 
\sqrt{n} 
\left( 
\frac{U_{i,n}^* Y_n\widetilde Q_n(\rho_i) R_{i,n}}{\sqrt{n}} ,
U_{i,n}^* ( Q_n(\rho_{i}) - m_n(\rho_{i}) I_N ) U_{i,n}, \right. 
\ \ \ \ \ \ \ \ \ \ \ \ \ \ \ \ \ \ \ \ \ \ \ \ \ \ \ \ \ \ \ \ 
\\ 
\left. 
\ \ \ \ \ \ \ \ \ \ \ \ \ \ \ \ \ \ \ \ \ \ \ \ \ \ \ \ \ \ \ \ 
\ \ \ \ \ \ \ \ \ \ \ \ \ \ \ \ \ \ \ \ \ \ \ \ \ \ \ \ \ \ \ \ 
R_{i,n}^* ( \widetilde Q_n(\rho_{i}) - \widetilde T_n(\rho_{i}) ) R_{i,n} 
\right)_{i=1}^p , 
\end{eqnarray*}
where  $P_n = U_n R_n^*$ is a sequence of factorization  such that $\Lambda_n$ satisfies the third 
item of Assumption \ref{fast-cvg}. We also write 
$U_n = [ U_{1,n} \, \cdots \, U_{t,n} ]$ and 
$R_n = [ R_{1,n} \, \cdots \, R_{t,n} ]$ where 
$U_{i,n} \in \C^{N \times j_i}$ and $R_{i,n} \in \C^{n \times j_i}$. 

This CLT is proven by using the Gaussian tools introduced in  Section 
\ref{gauss-tools}, namely the integration by parts  formula and the Poincar\'e-Nash inequality, and by relying on the invariance properties of the 
probability distribution of $n^{-1/2} X_n$. 
The fluctuations of the zeros of $\det\widehat S_n(x)$ outside 
$\support(\mu)$ are then deduced from this CLT by adapting the approach 
of \cite{bgm-fluct10} (Lemma \ref{le:chi_n}). \\

We now come to the basic mathematical tools needed for our proofs: 

\subsection{Differentiation formulas}
\label{diff-det} 
Let $\partial / \partial z = ( \partial/\partial x - \imath 
\partial / \partial y ) / 2$ and 
$\partial / \partial \bar z = ( \partial/\partial x + \imath 
\partial / \partial y ) / 2$ for $z = x + \imath y$. 
Given a Hermitian matrix $X$ with a spectral decomposition 
$X = \sum_\ell \lambda_\ell v_\ell v_\ell^*$, let 
$\adjugate(X) = \sum_k (\prod_{\ell\neq k} \lambda_\ell) v_k v_k^*$ 
be the classical adjoint of $X$, \emph{i.e.}, the transpose of its cofactor 
matrix. Let $\psi$ be a continuously differentiable real-valued function on 
$\R$. Then
\[
\frac{\partial \det \psi\left( n^{-1} {YY^*} \right)}{\partial \bar Y_{ij}}
= 
\frac 1n \left[ \adjugate\left( \psi\left( n^{-1}{YY^*} \right) \right) \, 
\psi'\left( n^{-1} {YY^*} \right) \, y_j \right]_i 
\] 
where $y_j$ is column $j$ of $Y$, 
see \cite[Lemma 3.9]{hlmnv-rmta11} for a proof. \\
We shall also need the expressions of the following derivatives of the elements
of the resolvents $Q$ and $\widetilde Q$ (see \cite{HLN08Ieee}):  
\[ 
\frac{\partial Q_{pq}}{\partial \bar Y_{ij} }=-\frac{1}{n}[QY]_{pj}Q_{iq},
\quad 
\frac{\partial \widetilde Q_{pq}}{\partial \bar Y_{ij} }
=-\frac{1}{n}\widetilde Q_{pj}[Y\widetilde Q]_{iq} . 
\] 

\subsection{Gaussian tools} 
\label{gauss-tools} 
Our analysis fundamentally relies on two mathematical tools which are 
often used in the analysis of large random matrices with Gaussian elements. 
The first is the so called integration by parts (IP) formula for functionals 
of Gaussian vectors introduced in random matrix theory in 
\cite{kho-pastur93,pastur-etal95}. 
Let $\Gamma : \R^{2Nn} \to \C$ be a continuously differentiable function 
polynomially bounded together with its partial derivatives. Then 
\[
\E \left(Y_{ij} \Gamma(Y)\right) =  d_j 
\EE\left[ \frac{\partial \Gamma(Y)}{\partial \bar Y_{ij}} \right]
\]
for any $i \in \{1,\ldots, N\}$ and $j \in \{ 1, \ldots, n \}$. 
The second tool is the Poincar\'e-Nash inequality (see for instance 
\cite{chen-jmva82}). In our situation, it states that the variance 
$\var( \Gamma(Y) )$ satisfies 
\[
\var( \Gamma(Y) ) \leq \sum_{i=1}^N \sum_{j=1}^n  d_j 
\E \left[ 
\left| \frac{\partial \Gamma({Y})}{\partial Y_{ij}} \right|^2 
+ 
\left| \frac{\partial \Gamma({Y})}{\partial \bar Y_{ij}} \right|^2 
\right] \ .
\]

We now recall the results of Silverstein and Choi \cite{sil-choi95} which
will be needed to prove Theorems \ref{1st-ord}-\eqref{no-spikes} and 
\ref{partout}. Close results can be found in \cite{mes-it08} and in 
\cite{mes-sp08}. 

\subsection{Analysis of the support of $\mu$} 
\label{subsec-jack}

\begin{proposition}[\cite{sil-choi95}, Th.1.1]
\label{lim-m(z)} 
For all $x \in \R^*$, $\displaystyle{\lim_{z\in\C_+ \to x} \mfat(z)}$ exists. 
The limit that we denote $\mfat(x)$ is continuous on $\R^*$. Moreover, $\mu$ 
has a continuous density $f$ on $\R^*$ given by $f(x) = \pi^{-1} \Im \mfat(x)$. 
\end{proposition} 
In \cite{sil-choi95}, the support of $\mu$ is also identified. Since 
$\mfat(z)$ is the unique solution in $\C_+$ of \eqref{m=f(m)} for 
$z\in \C_+$, it has a unique inverse on $\C_+$ given by 
\[
z(\mfat) = - \frac{1}{\mfat} + \int \frac{t}{1+c\mfat t} \nu(dt) 
\] 
The characterization of the support of $\mu$ is based on the following idea. 
On any open interval of $\support(\mu)^c$, $\mfat(x) = \int (t-x)^{-1} \mu(dt)$ 
is a real, continuous and increasing function. Consequently, 
it has a real, continuous and increasing inverse. In \cite{sil-choi95}, it is
shown that the converse is also true. More precisely, let 
$B = \{ \mfat : \mfat \neq 0, \, - (c\mfat)^{-1} \in \support(\nu)^c \}$, and 
let 
\begin{equation} 
\label{x(m)} 
\begin{array}{ccccl}
x &:& B & \longrightarrow & \R \\
& & \mfat &  \longmapsto & 
\displaystyle{x(\mfat) = 
- \frac{1}{\mfat} + \int \frac{t}{1+c\mfat t} \, \nu(dt)}  . 
\end{array} 
\end{equation} 
Then the following proposition holds: 
\begin{proposition}[\cite{sil-choi95}, Th. 4.1 and 4.2] 
\label{support} 
For any $x_0 \in \support(\mu)^c$, let $\mfat_0 = \mfat(x_0)$.  
Then $\mfat_0 \in B$, $x_0 = x(\mfat_0)$, and $x'(\mfat_0) > 0$. Conversely, 
let $\mfat_0 \in B$ such that $x'(\mfat_0) > 0$. Then 
$x_0 = x(\mfat_0) \in \support(\mu)^c$, and $\mfat(x_0) = \mfat_0$. 
\end{proposition}

The following proposition will also be useful: 
\begin{proposition}[\cite{sil-choi95}, Th. 4.4] 
\label{disjoint} 
Let $[\mfat_1, \mfat_2]$ and $[\mfat_3, \mfat_4]$ be two disjoint intervals of $B$ satisfying
$\forall \mfat \in (\mfat_1, \mfat_2) \cup (\mfat_3, \mfat_4)$, $x'(\mfat) > 0$. Then $[x_1, x_2]$ and
$[x_3, x_4]$ are disjoint where $x_i = x(\mfat_i)$. 
\end{proposition} 
The following result is also proven in \cite{sil-choi95}: 
\begin{proposition}
\label{weight}
Assume that $\nu(\{ 0 \}) = 0$. Then $\mu(\{ 0 \}) = \max(0, 1 - c^{-1})$. 
\end{proposition} 

We shall assume hereafter that $\nu(\{ 0 \}) = 0$ without loss of generality
(otherwise, it would be enough to change the value of $c$). 
The two following lemmas will also be needed: 
\begin{lemma}
\label{control-Ttilde} 
Let $\cal I$ be a compact interval of $\support(\mu)^c$, and let $D_{\cal I}$ 
be the closed disk having $\cal I$ as one of its diameters. Then there exists 
a constant $K$ which depends on $\cal I$ only such that 
\begin{gather*} 
\forall \, t \in \support(\nu), \ \forall \, z \in D_{\cal I}, \ 
\left| 1 + c \mfat(z) t \right| \geq K, \quad \text{and} \\  
\forall \, n \ \text{large enough}, \ 
\forall \, t \in \support(\nu_n), \ \forall \, z \in D_{\cal I}, \ 
\left| 1 + c_n m_n(z) t \right| \geq K . 
\end{gather*} 
\end{lemma} 
From the second inequality, we deduce that if $\{ 0 \} \not\in {\cal I}$,
the matrix function $\widetilde T_n(z)$ is analytic in a neighborhood of 
$\cal I$ for $n$ large enough, and 
\begin{equation}
\label{bound-Ttilde} 
\limsup_n \sup_{z \in D_{\cal I}} \| \widetilde T_n(z) \| < \infty . 
\end{equation} 
\begin{proof}
When $z \in \C_+$, $\Im \mfat(z) > 0$ and $\Im(-(c\mfat(z))^{-1}) > 0$, and we have
the opposite inequalities when $\Im z < 0$. 
Applying Proposition \ref{support} for $z \in \cal I$, we deduce that
$| \mfat(z) |$ and  
$f(z) = \bs d( - (c \mfat(z))^{-1}, \support(\nu) )$ are positive 
on $D_{\cal I}$. Since these functions are continuous on this compact set, 
$\min | \mfat(z) | = K_1 > 0$ and 
$\min f(z) = K_2 > 0$ on $D_{\cal I}$. Consequently, for any $z \in 
D_{\cal I}$ and any $t \in \support(\nu)$, 
$| 1 + c \mfat(z) t | = | c \mfat(z) ( - (c \mfat(z))^{-1} - t ) | 
\geq | c \mfat(z) | f(z) \geq c K_1 K_2 > 0$. We now prove the second inequality.
Denote by ${\bs d}_{\text{H}}(A, B)$ the Hausdorff distance between two sets 
$A$ and $B$. Let $f_n(z) = {\bs d}( - (c_n m_n(z))^{-1}, \support(\nu_n))$. 
We have 
\begin{align*} 
f_n(z) &\leq {\bs d}\Bigl( \frac{-1}{c_n m_n(z)}, \frac{-1}{c \mfat(z)} \Bigr) 
+  {\bs d}\Bigl( \frac{-1}{c \mfat(z)}, \support(\nu_n) \Bigr)  \\ 
&\leq {\bs d}\Bigl( \frac{-1}{c_n m_n(z)}, \frac{-1}{c \mfat(z)} \Bigr) 
+  f(z) + 
{\bs d}_{\text{H}}(\support(\nu_n),  \support(\nu) ) , 
\end{align*}  
and 
$f(z) \leq {\bs d}( -({c_n m_n(z)})^{-1}, -({c \mfat(z)})^{-1} ) 
+  f_n(z) + {\bs d}_{\text{H}}(\support(\nu_n),  \support(\nu) )$ 
similarly. Since $m_n(z)$ converges uniformly to $\mfat(z)$ and 
$\inf|\mfat(z)| > 0$ on $D_{\cal I}$, we get that 
${\bs d}( -({c_n m_n(z)})^{-1}, -({c \mfat(z)})^{-1} ) \to 0$
uniformly on this disk. By Assumption \ref{assump:D}, 
${\bs d}_{\text{H}}(\support(\nu_n),  \support(\nu) ) \to 0$. 
Hence $f_n(z)$ converges uniformly to $f(z)$ on 
$D_{\cal I}$ which proves the second inequality. 
\end{proof} 

\begin{lemma}
\label{lm:ST-Ttilde}
Assume the setting of Lemma \ref{control-Ttilde}, and assume that 
$\{ 0 \} \not\in {\cal I}$. Then for any sequence of vectors 
$\tilde u_n \in \C^n$ such that $\sup_n \| \tilde u_n \| < \infty$, the 
quadratic forms  
$\tilde u_n^* \widetilde T_n(z) \tilde u_n$ are the Stieltjes Transforms of
positive measures $\gamma_n$ such that $\sup_n \gamma_n(\R) < \infty$ and 
$\gamma_n(\cal I) = 0$ for $n$ large enough. 
\end{lemma} 
Indeed, one can easily check the conditions that enable 
$\tilde u_n^* \widetilde T_n(z) \tilde u_n$ to be a Stieltjes Transform 
of a positive finite measure. The last result is obtained by analyticity
in a neighborhood of $\cal I$. In fact, it can be checked that 
$\support(\gamma_n) \subset \support(\mu_n) \cup \{ 0 \}$.

\subsection{A control over the support of $\theta_n$} 

In this paragraph, we adapt to our case an idea developed in 
\cite{cdf09} to deal with Wigner matrices whose elements distribution
satisfies a Poincar\'e-Nash inequality. 

\begin{proposition}
\label{EQ-T}
For any sequence of $n \times n$ deterministic diagonal nonnegative matrices 
$\widetilde U_n$ such that $\sup_n \| \widetilde U_n \| < \infty$, 
\begin{gather*} 
\left| \frac 1n \tr \E Q_n(z) -  m_n(z) \right| \leq 
\frac{P(|z|) R(|\Im(z)|^{-1})}{n^2}, \ \text{and} \\
\left| \frac 1n \tr \widetilde U_n \E \widetilde Q_n(z) -  
\frac 1n \tr \widetilde U_n \widetilde T_n(z)  
\right| \leq 
\frac{P(|z|) R(|\Im(z)|^{-1})}{n^2} 
\end{gather*} 
for $z \in \C_+$, where $P$ and $R$ are polynomials with nonnegative 
coefficients independent of $n$. 
\end{proposition}
This proposition is obtained from a simple extension of the results of 
\cite[Th. 3 and Prop.5]{HLN08Ieee} from $z \in (-\infty, 0)$ to $z\in \C_+$.
\\
The following important result, due to Haagerup and Thorbj\o rnsen, is
established in the proof of \cite[Th.6.2]{haagerup}: 
\begin{lemma}
\label{lm-haagerup}
Assume that $h(z)$ is an analytic function on $\C_+$ that satisfies
$|h(z)| \leq P(|z|) R(|\Im(z)|^{-1})$ where $P$ and $R$ are polynomials with
nonnegative coefficients. Then for any function
$\varphi\in {\cal C}_c^\infty(\R,\R)$, the set of smooth real-valued functions 
with compact support in $\R$, 
\[
\limsup_{y\downarrow 0} \left| \int_\R \varphi(x) h(x +\imath y) dx \right| 
< \infty .
\]
\end{lemma}

Since $N^{-1} \tr Q_n(z)$ is the Stieltjes Transform of the spectral measure
$\theta_n$, the inversion formula \eqref{perron} shows that 
\[
\int \varphi(t) \, \theta_n(dt) = 
\frac 1\pi \lim_{y\downarrow 0} 
\Im \int \varphi(x) \frac 1N \tr Q_n (x+ \imath y) \, dx 
\]
for any function $\varphi \in {\cal C}^\infty_c(\R, \R)$. 
Using then Proposition \ref{EQ-T} and Lemma \ref{lm-haagerup}, we obtain the 
following result:
\begin{proposition}
\label{prop-bound-nbr-ev}
For any function $\varphi \in {\cal C}_c^\infty(\R,\R)$,
\[
\left| \EE \int \varphi(t) \, \theta_n(dt) \ - \ 
\int \varphi(t) \, \mu_n(dt) \right| \leq 
\frac{K}{n^2} . 
\]
\end{proposition}

\section{Proofs of first order results}
\label{1st-order}

In all this section, $\cal I$ is a compact interval of a component $(a,b)$ of
$\support(\mu)^c$, and $z$ is a complex number such that $\Re(z) \in \cal I$ 
and $\Im(z)$ is arbitrary. Moreover, $u_n, v_n \in \C^N$ and 
$\tilde u_n, \tilde v_n \in \C^n$ are sequences of deterministic vectors such 
that 
$\sup_n \max( \| u_n \|, \| v_n \|, \| \tilde u_n \|, \| \tilde v_n \| ) < 
\infty$, and $\widetilde U_n$ is a sequence of $n \times n$ diagonal 
deterministic matrix such that $\sup_n \| \widetilde U_n \| < \infty$. \\ 
We now introduce the regularization function alluded to in the introduction. 
Choose $\varepsilon > 0$ small enough so that 
$\cal I \cap {\cal S}_{\varepsilon} = \emptyset$ where 
${\cal S}_\varepsilon = \{ x \in \R, \bs d(x, \support(\mu) \cup \{ 0 \}) 
\leq \varepsilon \}$. 
Fix $0 < \varepsilon' < \varepsilon$, let $\psi : \R \to [0,1]$ be a
continuously differentiable function such that
\[
\psi(x) = \left\{\begin{array}{ll}  
1 & \text{if} \ x \in {\cal S}_{\varepsilon'} \\ 
0 & \text{if} \ x \in \R - {\cal S}_{\varepsilon}   
\end{array} \right. 
\]
and let $\phi_n = \det \psi(n^{-1} Y_n Y_n^*)$. 
In all the subsequent derivations, quantities such as $u_n^* Q_n(z) u_n$ or
$\tilde u_n^* \widetilde Q_n(z) \tilde u_n$ for $\Re(z) \in \cal I$ will 
be multiplied by $\phi_n$ in 
order to control their magnitudes when $z$ is close to the real axis. By 
performing this regularization as is done in \cite{hlmnv-rmta11}, we shall be 
able to define and control the moments of random variables such as 
$\phi_n \, u_n^* Q_n(z) u_n$ or $\phi_n \, \tilde u_n^* \widetilde Q_n(z) \tilde u_n$
with the help of the Gaussian tools introduced in Section \ref{gauss-tools}. 
\\

We start with a series of lemmas. 
The first of these lemmas relies on Proposition \ref{prop-bound-nbr-ev} and
on the Poincar\'e-Nash inequality. Its detailed proof is a minor modification
of the proof of \cite[Lemma 3]{hlmnv-rmta11} and is therefore omitted: 
\begin{lemma}
\label{lm-noeig}
Given $0 < \varepsilon' < \varepsilon$, let $\varphi$ be a smooth nonnegative
function equal to zero on ${\cal S}_{\varepsilon'}$ and to one on 
$\R - {\cal S}_{\varepsilon}$. Then for any $\ell \in \N$, there exists a 
constant $K_\ell$ for which 
\[
\E\left[ \left(\tr \varphi(n^{-1} Y_n Y_n^*) \right)^\ell \right] \leq 
\frac{K_\ell}{n^\ell} . 
\]
\end{lemma}
\begin{remark}
Notice that this lemma proves Theorem \ref{lsm}-\eqref{noeig-sil}. The proof 
provided in \cite{BaiSil98} is in fact more general, being not restricted to 
the Gaussian case. 
\end{remark} 

\begin{lemma} 
\label{mom-phi} 
For any $\ell \in \N$, the following holds true:
\[
\E \left[ \left( \sum_{i,j=1}^{N,n}  d_j  
\left| 
\frac{\partial \phi_n}{\partial \bar Y_{ij}} 
\right|^{2} \right)^\ell \right] 
\leq \frac{K_\ell}{n^{2\ell}} . 
\] 
\end{lemma} 
\begin{proof} 
Letting $n^{-1/2} Y = W \diag(\sqrt\lambda_1,\cdots, \sqrt\lambda_N) V^*$ 
be a singular value decomposition of $n^{-1/2} Y$, we have
\[
\adjugate\left( \psi\Bigl( \frac{YY^*}{n} \Bigr) \right)
\psi'\Bigl( \frac{YY^*}{n} \Bigr) \frac{Y}{\sqrt{n}} = 
W \Xi V^* \ \text{where} \ 
\Xi = \diag\Bigl( \sqrt\lambda_k \psi'(\lambda_k) \prod_{\ell\neq k} 
\psi(\lambda_\ell) \Bigr)_{k=1}^N  
\] 
and we observe that $\tr \Xi^2 \leq K Z_n$ where 
$Z_n = \sharp\{ k \, : \, \lambda_k \in {\cal S}_\varepsilon - 
{\cal S}_{\varepsilon'} \}$. Using the first identity in Section \ref{diff-det} 
and recalling that $|\tr(AB)| \leq \| A \| \tr B$ when $A$ is a square matrix and $B$
is a Hermitian nonnegative matrix, we have 
\[ 
\E \left[ \left( \sum_{i,j=1}^{N,n}  d_j  
\left| \frac{\partial \phi_n}{\partial \bar Y_{ij}} \right|^2 \right)^\ell
\right]  
= \frac{1}{n^\ell} 
\E \left[ \left( \tr \left(\adjugate(\psi) \psi' \frac{Y  D Y^*}{n} 
\adjugate(\psi) \psi' \right) \right)^\ell \right] 
\leq \frac{K}{n^\ell} \E Z_n^\ell
\] 
and the result follows from Lemma \ref{lm-noeig} with a proper choice of 
$\varphi$.  
\end{proof}

\begin{lemma}
\label{lm-var}
The following inequalities hold true: 
\begin{gather*} 
\E\left| \phi_n \, u_n^* Q_n(z) v_n - 
\E [ \phi_n \, u_n^* Q_n(z) v_n ] \right|^4 \leq \frac{K}{n^2}, 
\\ 
\E\left| \phi_n \, \tilde u_n^* \widetilde Q_n(z) \tilde v_n - 
\E [ \phi_n \, \tilde u_n^* \widetilde Q_n(z) \tilde v_n ] \right|^4 
\leq \frac{K}{n^2} , \\ 
\var \left( \phi_n \tr Q_n(z) \right) \leq K . 
\end{gather*} 
\end{lemma}
\begin{proof}
We shall only prove the first inequality. 
By the polarization identity, this inequality is shown whenever we show that 
$\E\left| \phi \, u^* Q u - \E [ \phi \, u^* Q u ] \right|^4 \leq K / n^2$.
Let us start by showing that $\var (\phi \, u^* Q u ) \leq K / n$. By the 
Poincar\'e-Nash inequality, we have 
\begin{align*}
\var\left( \phi u^* Q u \right)& \leq 
2 \sum_{i,j=1}^{N,n}  d_j \E \left| 
\frac{\partial \phi u^* Q u}{\partial \bar Y_{ij}} \right|^2 \\ 
&\leq 
4 \sum_{i,j=1}^{N,n}  d_j  
\E \left| \phi \frac{\partial u^* Q u}{\partial \bar Y_{ij}} \right|^2 
+
4 \sum_{i,j=1}^{N,n}  d_j  
\E \left| u^*Qu \frac{\partial \phi}{\partial \bar Y_{ij}} \right|^2 . 
\nonumber 
\end{align*}
Using the expression of $\partial Q_{pq} / \partial \bar Y_{ij}$ 
in Section \ref{diff-det}, we have 
\[
\partial u^* Q u / \partial \bar Y_{ij} = - n^{-1} u^* Q y_j \, 
[ Qu]_i ,
\] 
hence 
\[
\sum_{i,j=1}^{N,n}  d_j  
\E \left| \phi \frac{\partial u^* Q u}{\partial \bar Y_{ij}} \right|^2 
= \frac 1n \E \left[ 
\phi^2 \, u^* Q \frac{Y  D Y^*}{n} Q u \, u^* Q^2 u \right] 
\leq \frac Kn  
\] 
since the argument of the expectation is bounded for $\Re(z) \in \cal I$. 
From the first identity in Section \ref{diff-det}, 
$\sum_{i,j}  d_j  
\E \left| u^*Qu \partial \phi / \partial \bar Y_{ij} \right|^2
\leq K 
\sum_{i,j}  d_j  
\E \left| \partial \phi / \partial \bar Y_{ij} \right|^2$ which is bounded
by $K/n^2$ by Lemma \ref{mom-phi}.  
It results that $\var (\phi \, u^* Q u ) \leq K / n$. 
Now, writing $\stackrel{\circ}{X} = X - \E X$, 
\[ 
\E\Bigl| \overbrace{\phi \, u^* Q u}^\circ \Bigr|^4 = 
\left( \var ( \phi \, u^* Q u ) \right)^2
+ \var \Bigl( \Bigl(\overbrace{\phi \, u^* Q u}^\circ \Bigr)^2 \Bigr)  
\leq K/n^2 + \var \Bigl( \Bigl(\overbrace{\phi \, u^* Q u}^\circ \Bigr)^2 
\Bigr) . 
\] 
By the Poincar\'e-Nash inequality, 
\begin{multline*} 
\var \Bigl( \Bigl(\overbrace{\phi \, u^* Q u}^\circ \Bigr)^2 \Bigr) \leq 
2 \sum_{i,j=1}^{N,n} 
 d_j \E \Bigl| 
\partial \Bigl(\overbrace{\phi \, u^* Q u}^\circ\Bigr)^2 \, / \, 
\partial \bar Y_{ij} \Bigr|^2  \\
\leq 
16 \sum_{i,j=1}^{N,n} 
 d_j \E \Bigl| 
\overbrace{\phi \, u^* Q u}^\circ \, \phi \, 
\frac{\partial u^* Q u}{\partial \bar Y_{ij}} \Bigr|^2  
+
16 \sum_{i,j=1}^{N,n} 
 d_j \E \Bigl| 
\overbrace{\phi \, u^* Q u}^\circ \, u^* Q u \, 
\frac{\partial \phi}{\partial \bar Y_{ij}} \Bigr|^2  
:= V_1 + V_2 . 
\end{multline*} 
By developing the derivative in $V_1$ similarly to above, 
$V_1 \leq Kn^{-1} \E\Bigl| \overbrace{\phi \, u^* Q u}^\circ \Bigr|^2 \leq 
K n^{-2}$. By the Cauchy-Schwarz inequality and Lemma \ref{mom-phi}, 
\[
V_2 \leq 
K \sum_{i,j=1}^{N,n} 
 d_j \E \Bigl| 
\overbrace{\phi \, u^* Q u}^\circ  
\frac{\partial \phi}{\partial \bar Y_{ij}} \Bigr|^2  
\leq
\frac{K}{n^2} \Bigl( \E\Bigl| \overbrace{\phi \, u^* Q u}^\circ \Bigr|^4 
\Bigr)^{1/2}. 
\]
Writing $a_n = n^2 \E\Bigl| \overbrace{\phi \, u^* Q u}^\circ \Bigr|^4$, we
have shown that $\sqrt{a_n} \leq K / \sqrt{a_n} + K / n$. Assume that 
$a_n$ is not bounded. Then there exists a sequence $n_k$ of integers such 
that $a_{n_k} \to \infty$, which raises a contradiction. The first inequality
in the statement of this lemma is shown. The other two inequalities can be
shown similarly. 
\end{proof} 

\begin{lemma}
\label{1-E[phi]}
The following holds true: 
\[
1 - \EE\phi_n \leq \frac{K_\ell}{n^\ell} \quad \text{for any } 
\ell \in \N .  
\]
\end{lemma}
\begin{proof}
For $0 < \varepsilon_1 < \varepsilon'$ where $\varepsilon'$ is defined 
in the construction of $\psi$,  let $\varphi$ be a smooth nonnegative
function equal to zero on ${\cal S}_{\varepsilon_1}$ and to one on 
$\R - {\cal S}_{\varepsilon'}$. Then 
$1-\phi_n \leq \left( \tr(\varphi(n^{-1} YY^*)) \right)^\ell$ for any 
$\ell \in \N$, and the result stems from Lemma \ref{lm-noeig}. 
\end{proof}
\begin{lemma}
\label{mean-tr(Q)} 
The following inequalities hold true (recall that $\Re(z) \not\in 
\support(\mu)$): 
\[ 
\left| \E\left[ \phi_n \tr Q_n(z) \right] - N m_n(z) 
\right| \leq \frac{K}{n}, \quad \text{and} \quad 
\left| \tr \widetilde U_n \left( \EE [ \phi_n \widetilde Q_n(z) ] 
- \widetilde T_n(z) \right) \right| \leq \frac{K}{n}. 
\]
\end{lemma} 
\begin{proof} 
Let $\varepsilon$ be defined in the construction of $\psi$. 
Choose a small $\varepsilon_1 > \varepsilon$ in such a way that 
${\cal S}_{\varepsilon_1} \cap \cal I = \emptyset$. 
Let $\zeta$ be a ${\cal C}_c^\infty(\R,\R)$ nonnegative function equal to 
one on ${\cal S}_{\varepsilon}$ and to zero on
$\R - {\cal S}_{\varepsilon_1}$, so that 
\[
\phi \frac 1n \tr Q = \phi \int \frac{\zeta(t)}{t-z} \theta_n(dt) \ . 
\]
Using this equality, and recalling that $\phi \in [0,1]$, we have
\[
\left| \E \phi \frac 1n \tr Q - 
    \E \int \frac{\zeta(t)}{t-z} \theta_n(dt) \right| \leq 
\E\left[ (1-\phi) \left| \int \frac{\zeta(t)}{t-z} \theta_n(dt) \right| 
\right] 
\leq \frac{1-\E \phi}{{\bs d}(z, {\cal S}_{\varepsilon_1})} 
\leq \frac{K_\ell}{n^\ell}  
\] 
for any $\ell \in \N$. Moreover, we have 
\[
\left| \E \int \frac{\zeta(t)}{t-z} \theta_n(dt) 
- m_n(z) \right| = 
\left|  \E \int \frac{\zeta(t)}{t-z} \theta_n(dt) 
 - \int \frac{\zeta(t)}{t-z} \mu_n(dt) \right| 
\leq \frac{K}{n^2} 
\] 
by Proposition \ref{lm-haagerup}, and the first inequality is proved. \\
By performing a spectral factorization of $n^{-1} Y^*Y$, one can check that
$n^{-1} \tr \widetilde U \widetilde Q(z)$ is the Stieltjes Transform of a
positive measure $\tau_n$ such that $\sup_n \tau_n(\R) < \infty$ and 
$\support(\tau_n) \subset \support(\theta_n) \cup \{ 0 \}$. 
By Lemma \ref{lm:ST-Ttilde}, 
$n^{-1} \tr \widetilde U \widetilde T(z)$ is the Stieltjes Transform of a 
positive measure $\gamma_n$ such that $\sup_n \gamma_n(\R) < \infty$ and 
$\gamma_n(\cal I) = 0$ for all large $n$. With the help of the second 
inequality of Proposition \ref{EQ-T}, we have a result similar to that of 
Proposition \ref{prop-bound-nbr-ev}, namely that 
$| \E \int \varphi d\tau_n - \int \varphi d\gamma_n | \leq K / n^2$ 
for any function $\varphi \in {\cal C}_c^\infty(\R,\R)$. 
We can then prove the second inequality similarly to the first one. 
\end{proof} 

\begin{lemma}
\label{mean-fq} 
The following inequalities hold true:
\begin{align*}
\left| \E[ \phi_n \, u_n^* Q_n(z) v_n ] - u_n^* v_n \, m_n(z) \right| &\leq K/n, \\
\left| \E[ \phi_n \, \tilde u_n^* \widetilde Q_n(z) \tilde v_n ] - \tilde u_n^* \widetilde T_n(z) \tilde v_n \right| &\leq K/n.
\end{align*}
\end{lemma} 
In \cite{hlnv-bilinear-ihp11}, it is proven in a more general setting 
that $\left| \E u_n^* Q_n(z) u_n - \| u_n\|^2 m_n(z) \right| \leq P(|z|) 
R(| \Im(z) |^{-1}) / \sqrt{n}$ for any $z \in \C_+$. Observing that 
$u_n^* Q_n(z) u_n$ and $\| u_n\|^2 m_n(z)$
are Stieltjes Transforms of positive measures, and mimicking the proof of
the previous lemma, we can establish this lemma with the 
rate ${\cal O}(n^{-1/2})$, which is in fact enough for our purposes. However, 
in order to give a flavor of the derivations that will be carried out 
in the next section, we consider here another proof that uses the IP  
formula and the Poincar\'e-Nash inequality. To that end, we introduce new 
notations: 
\begin{gather*} 
\beta(z) = \phi_n \frac 1n \tr Q_n(z), \ \alpha(z) = \E\beta(z), \ 
\hat\beta(z) = \beta(z) - \phi_n \alpha(z), \quad \text{and} \\
\tilde\alpha(z) = \frac 1n \tr  D_n 
[ - z (I_n + \alpha(z)  D_n) ]^{-1} . 
\end{gather*} 
\begin{proof} 
We start with the first inequality. By the IP formula, we have
\[
\E[ Q_{pi} Y_{ij} \bar Y_{\ell j} \phi ] = 
- \frac{ d_j}{n} \E [ [Q y_j]_p Q_{ii} \bar Y_{\ell j} \phi ] 
+ \delta(\ell -i)  d_j \E[Q_{pi} \phi] 
+ \frac{ d_j}{n} \E[ Q_{pi} \bar Y_{\ell j} 
[\adjugate(\psi) \psi' y_j]_i ] .
\]   
Taking the sum over $i$, we obtain
\begin{align*} 
\E[ [Q y_j]_p \bar Y_{\ell j} \phi ] &= 
-  d_j \E [ [Q y_j]_p \bar Y_{\ell j} \beta ] 
+  d_j \E[Q_{p\ell} \phi] 
+ \frac{ d_j}{n} \E[ \bar Y_{\ell j} 
[Q \adjugate(\psi) \psi' y_j]_p ] . 
\end{align*}
Writing $\beta = \hat\beta + \phi\alpha$, we get 
\begin{align*} 
\E[ [Q y_j]_p \bar Y_{\ell j} \phi ] &= 
 \frac{ d_j}{1+\alpha  d_j}  \E[Q_{p\ell} \phi]
- \frac{ d_j}{1+\alpha  d_j}  
\E [ [Q y_j]_p \bar Y_{\ell j} \hat\beta] \\
&+ \frac{ d_j}{n(1+\alpha  d_j)}  
\E[ \bar Y_{\ell j} 
[Q \adjugate(\psi) \psi' y_j]_p ] . 
\end{align*} 
Taking the sum over $j$, we obtain
\begin{align*} 
\E\left[ \left[ Q \frac{YY^*}{n} \right]_{p\ell} \phi \right] 
&= 
- z \tilde\alpha \E[ Q_{p\ell} \phi ] - \E\left[ \hat\beta 
\left[ Q \frac{Y D (I+\alpha  D)^{-1} Y^*}{n} 
\right]_{p\ell} \right] \\
& + \frac{1}{n} 
\E\left[ 
\left[ Q \adjugate(\psi) \psi' 
\frac{Y  D (I+\alpha  D)^{-1} Y^*}{n} 
\right]_{p\ell} \right] .
\end{align*} 
We now use the identity $z Q = n^{-1} Q YY^* - I$, which results in
\begin{align*}
z \E[ Q_{p\ell} \phi] &= \E\left[ \left[ Q \frac{YY^*}{n} \right]_{p\ell} 
\phi \right] - \delta(p-\ell) \E[\phi],  \\ 
\E[ Q_{p\ell} \phi] &= 
\frac{\delta(p-\ell)}{- z(1+ \tilde\alpha)} \E[\phi] + \frac{\text{2nd and 3rd terms of next to last equation}}
{z(1+ \tilde\alpha)} . 
\end{align*} 
Multiplying each side by $[u^*]_p [v]_\ell$ and taking the sum over $p$ and 
$\ell$, we finally obtain
\begin{align}
\E[ u^* Q v \phi ] &= 
 \E[\phi]
\frac{u^* v}{- z(1+ \tilde\alpha)} - [-z(1+ \tilde\alpha )]^{-1}\E\left[ \hat\beta 
u^* Q \frac{Y D (I+\alpha  D)^{-1} Y^*}{n} 
 v  \right] \nonumber \\
&
+\frac 1n  [-z(1+ \tilde\alpha )]^{-1} \E\left[ u^* Q  
\adjugate(\psi) \psi' 
\frac{Y  D (I+\alpha  D)^{-1} Y^*}{n}
v  \right]  .
\label{eq-EuQu} 
\end{align} 
Let us evaluate the three terms at the right hand side of this equality. 
From Lemma \ref{mean-tr(Q)}, we have $\alpha = c_n m_n + {\cal O}(n^{-2})$. 
Using in addition the bound \eqref{bound-Ttilde}, we obtain 
$\tilde\alpha = n^{-1} \tr( D(- z(I+ c_n m_n  D+ 
(\alpha - c_n m_n)  D)^{-1} )= n^{-1} \tr  D \widetilde T
+ {\cal O}(n^{-2})$. Since 
$m_n(z) = ( -z ( 1 + n^{-1} \tr  D \widetilde T(z) ) )^{-1}$, we 
obtain that $\left(-z(1+ \tilde\alpha )\right)^{-1} 
= m_n(z) + {\cal O}(n^{-2})$.
Using in addition Lemma \ref{1-E[phi]}, we obtain that the first right hand
side term of \eqref{eq-EuQu} is  $u^* v \, m_n(z) + {\cal O}(n^{-2})$. 
Due to the presence of $\phi$ in the expression of $\hat\beta$, the second
term is bounded by $K \EE|\hat\beta|$.
Moreover, 
$\hat\beta = n^{-1} \phi \tr Q - n^{-1} \E[\phi \tr Q] + 
(1-\phi) n^{-1} \E[\phi \tr Q]$. By Lemmas \ref{1-E[phi]} and \ref{lm-var}, 
$\E | \hat\beta| = {\cal O}(n^{-1})$. 
The third term can be shown to be bounded by $ K n^{-1} 
\E \tr \varphi(n^{-1} YY^*) = {\cal O}(n^{-2})$ where $\varphi$ is as in the 
statement of Lemma \ref{lm-noeig}. This proves the first inequality in 
the statement of the lemma. \\ 
The second result in the statement of the lemma is proven similarly. The
proof requires the second inequality of Lemma 
\ref{mean-tr(Q)}.
\end{proof} 
The proof of the following lemma can be done along the same lines and 
will be omitted: 
\begin{lemma}
\label{uXQv}
The following inequalities hold true:
\begin{align*}
\left| \E \phi_n u^*_n Y_n \widetilde Q_n(z) \tilde v_n \right|
&\leq K / \sqrt{n} \\
\E \left| \phi_n u^*_n Y_n \widetilde Q_n(z) \tilde v_n \right|^4
&\leq K. 
\end{align*}
\end{lemma} 

We now prove Theorem \ref{1st-ord}-\eqref{loc-spikes}. 

\subsection*{Proof of Theorem \ref{1st-ord}-(\ref{loc-spikes})}   

Our first task is to establish the properties of $H_*(z)$ given in the 
statement of Theorem \ref{1st-ord}-(\ref{loc-spikes}). 
By Lemma \ref{control-Ttilde} and the fact that $\support(\Lambda_*) 
\subset \support(\nu)$, the function $H_*(z)$ can be analytically extended
to $(a,b)$. 
The comments preceding Theorem \ref{1st-ord} show that any $H_*(z)$ is 
the Stieltjes Transform of a matrix-valued nonnegative measure $\Gamma$. 
Since $H_*(z)$ is analytic on $(a,b)$, it is increasing on this interval in 
the order of Hermitian matrices, and the properties of this function given 
in the statement of Theorem \ref{1st-ord}-\eqref{loc-spikes} are established. 

We now prove the convergence stated in Theorem 
\ref{1st-ord}-\eqref{loc-spikes}, formalizing the argument introduced in 
Section \ref{idee-preuve}. In the remainder, we restrict ourselves to 
the probability one set where $n^{-1} Y Y^*$ has no eigenvalues for large $n$ 
in a large enough closed interval in $(a,b)$. Given a compact 
interval $\cal I \subset (a,b)$, let $D_{\cal I}^\circ$ be the open disk with 
diameter $\cal I$. Observe that the functions $\widehat S_n(x)$ and $S_n(x)$ 
introduced in Section \ref{idee-preuve} can be extended analytically to 
a neighborhood of $D_{\cal I}$, and the determinants of these functions do
not cancel outside the real axis. Let 
$\widehat L_n = \sharp\{ i \, : \, \hat\lambda_i^n \in 
D_{\cal I}^\circ \}$, $L_n = \sharp\{\text{ zeros of } \det S_n(z) \text{ in } 
D_{\cal I}^\circ \}$, and 
$L = \sharp\{ i \, : \, \rho_i \in D_{\cal I}^\circ \}$. 
We need to prove that $\widehat L_n = L$ with probability one for $n$ large. 
By the argument of \cite{ben-rao-published, ben-rao-11}, $\widehat L_n$ is
equal to the number of zeros of $\det\widehat S_n(z)$ in $\cal I^\circ$. 
By the well known argument principle for holomorphic functions,  
\begin{align*} 
\widehat L_n &= 
\frac{1}{2\imath \pi} \oint_{\partial D^\circ_{\cal I}} 
\!\!
\frac{(\det \widehat S_n(z))'}{\det \widehat S_n(z)} \, dz, \\ 
L_n &= 
\frac{1}{2\imath \pi} \oint_{\partial D^\circ_{\cal I}} 
\!\!
\frac{(\det S_n(z))'}{\det S_n(z)} \, dz 
= 
\frac{1}{2\imath \pi} \oint_{\partial D^\circ_{\cal I}} 
\!\!
\frac{(\det( H_n(z) + I_r) )'}{\det( H_n(z) + I_r) } \, dz \quad 
\text{and} \\ 
L &= 
\frac{1}{2\imath \pi} \oint_{\partial D^\circ_{\cal I}} 
\!\!
\frac{(\det (H_*(z) + I_r))'}{\det(H_*(z) + I_r)} \, dz 
\end{align*}  
where $\partial D^\circ_{\cal I}$ is seen as a positively oriented contour. \\
For any $1\leq k, \ell \leq r$, let $h_{n, k,\ell}(z) = 
[U_n^* (Q_n(z) - m_n(z) I_N) U_n]_{k,\ell}$. Let $V$ be a small neighborhood
of $D_{\cal I}$, the closure of $D^\circ_{\cal I}$. Let $z_m$ be a 
sequence of complex numbers in $V$ having an accumulation point in $V$. 
By Lemmas \ref{lm-noeig}, \ref{lm-var} and \ref{mean-fq} and the Borel
Cantelli lemma, $h_{n, k,\ell}(z_m) \toasshort 0$ as $n\to\infty$ for every 
$m$. Moreover, for $n$ large, the $h_{n,k,\ell}$ are uniformly bounded on any
compact subset of $V$. By the normal family theorem, every $n$-sequence 
of $h_{n,k,\ell}$ contains a further subsequence which converges uniformly on
the compact subsets of $V$ to a holomorphic function $h_*$. Since 
$h_*(z_m)= 0$ for every $m$, we obtain that almost surely, $h_{n,k,\ell}$ 
converges uniformly to zero on the compact subsets of $V$, and the same 
can be said about $\| U_n^* (Q_n(z) - m_n(z) I_N) U_n \|$. Using in addition
Lemmas \ref{control-Ttilde} and \ref{uXQv} we obtain the same result for  
$\| R_n^* (\widetilde Q_n(z) - \widetilde T_n(z)) R_n \|$ and 
$n^{-1/2} \| U_n^* Y_n \widetilde Q_n(z) R_n \|$. 

Since $\det X$ is a polynomial in the elements of matrix $X$, 
$\det \widehat S_n(z) - \det S_n(z)$ converges almost surely to zero on 
$\partial D^\circ_{\cal I}$, and this convergence is uniform. 
By analyticity, the same can be said about the derivatives of these quantities.
Moreover, $\det S_n(z)$ converges to $(-1)^r \det( H_*(z) + I_r)$ 
(which is the same for all accumulation points $\Lambda_*$) 
uniformly on $\partial D^\circ_{\cal I}$. Similarly, 
$(\det S_n(z))'$ converges to $(-1)^r (\det( H_*(z) + I_r))'$ 
uniformly on $\partial D^\circ_{\cal I}$. Furthermore, by construction of 
the interval $\cal I$, we have  
$\inf_{z \in \partial D^\circ_{\cal I}} | \det( H_*(z) + I_r)  | > 0$ 
which implies that 
$\liminf_n \inf_{z \in \partial D^\circ_{\cal I}} 
| \det S_n(z) | > 0$. It follows that $\widehat L_n = L_n$ and 
$L_n = L$ for $n$ large enough. 
This concludes the proof of Theorem \ref{1st-ord}-\eqref{loc-spikes}. 

\subsection*{Proofs of Theorems \ref{1st-ord}-(\ref{no-spikes}) and 
\ref{partout}}   
We start with the following lemma: 
\begin{lemma}
\label{m(1+ctm)} 
Let $A = \inf( \support(\mu) - \{ 0 \})$ and let $(a,b)$ be a component of 
$\support(\mu)^c$. Then the following facts hold true: 
\begin{itemize}
\item[(i)] If $b \leq A$, then $\mfat(x)(1+c\mfat(x)t)^{-1} > 0$
for all $x \in (a,b)$ and all $t \in \support(\nu)$. 
\item[(ii)] Alternatively, if $a > A$, then there exists a Borel set 
$E \subset \R_+$ such that $\nu(\partial E) = 0$ and 
\[
q(x) = \int_E \frac{\mfat(x)}{1+c\mfat(x)t} \nu(dt) < 0
\]
for all $x \in (a,b)$.     
\end{itemize} 
\end{lemma} 

\begin{proof} 
The proof is based on the results of Section \ref{subsec-jack}. 
To have an illustration of some of the proof arguments, the reader may refer 
to Figures \ref{fig:c<1} and \ref{fig:c>1} which provide typical plots of 
$x(\mfat)$ for $c<1$ and $c>1$ respectively. \\
We start by fixing a point $x_0$ in $(a,b)$, we write $\mfat_0 = \mfat(x_0)$ 
and we choose in the remainder of the proof $E = [0, - (c\mfat_0)^{-1}]$ with 
the convention $E = \emptyset$ when $\mfat_0 > 0$. We already assumed that 
$\nu(\{0\}) = 0$ in Section \ref{subsec-jack}. Since 
$- (c\mfat_0)^{-1} \in \support(\nu)^c$ by 
Proposition \ref{support}, $\nu(\{- (c\mfat_0)^{-1}\}) = 0$. Hence 
$\nu(\partial E) = 0$. To prove the lemma, we shall show that  
\begin{equation}
\label{nu-A} 
\nu(E) > 0 \Leftrightarrow a > A . 
\end{equation} 
To see why \eqref{nu-A} proves the lemma, consider first $a > A$. Then
$\mfat_0 < 0$ since $\nu(E) > 0$. Moreover $1 + c\mfat_0 t \geq 0$ for any 
$t \in E$. 
It results that $q(x_0) < 0$. Consider now another point $x_1 \in (a,b)$, and 
let $\mfat_1 = \mfat(x_1)$ and $E_1 = [0, - (c\mfat_1)^{-1}]$. By the same 
argument as for $x_0$, we get that 
$\int_{E_1} \mfat(x)(1+c\mfat(x)t)^{-1} \nu(dt) < 0$. But Proposition 
\ref{support} shows that the closed interval between $\mfat_0$ and $\mfat_1$ 
belongs to the set $B$.
It results that $\nu(E \bigtriangleup E_1) = 0$ where $E \bigtriangleup E_1$
is the symmetric difference between $E$ and $E_1$. Hence $q(x_1) < 0$ and 
$(ii)$ is true. \\ 
Assume now that $b \leq A$. If $\mfat_0 > 0$, then for all $x_1 \in (a,b)$,
$\mfat_1 = \mfat(x_1) > 0$ since the segment between $\mfat_0$ and $\mfat_1$
belongs to $B$, and $(i)$ is true. Assume that $\mfat_0 < 0$. Then since 
$\nu(E) = 0$, for any $t \in \support(\nu)$, $t > -(c \mfat_0)^{-1}$, hence 
$\mfat_0 (1 + c \mfat_0 t)^{-1} > 0$. 
If we take another point $x_1 \in (a,b)$, then the associated
set $E_1$ will also satisfy $\nu(E_1) = 0$ since 
$\nu(E \bigtriangleup E_1) = 0$. Hence we also have 
$\mfat_1 (1 + c \mfat_1 t)^{-1} > 0$ for any $t \in \support(\nu)$, which
proves $(i)$. \\ 
We now prove \eqref{nu-A} in the $\Leftarrow$ direction, showing that 
$x_0 > A \Rightarrow \nu(E) > 0$.  
Since $\mfat(z)$ is the Stieltjes Transform of a probability measure supported 
by $\R_+$, the function $\mfat(x)$ decreases to zero as $x\to-\infty$. 
Furthermore, $(0,\infty)$ belongs to $B$. Hence, by Proposition \ref{support}, 
$x_0 > A \Rightarrow \mfat_0 < 0$. 
Assume that $\nu(E) = 0$. Then $(-\infty, \mfat_0] \subset B$. Since 
$t > - (c\mfat_0)^{-1}$ in the integral in 
\eqref{x(m)}, $x(\mfat) \to 0$ as $\mfat\to-\infty$ by the dominated 
convergence theorem. By Propositions \ref{lim-m(z)}, \ref{support} and 
\ref{disjoint}, $x(\mfat)$ should be increasing from $0$ to $x_0$ on 
$(-\infty, \mfat_0]$. This contradicts $x_0 > A$. \\ 
We now prove \eqref{nu-A} in the $\Rightarrow$ direction. To that end, 
we consider in turn the cases $c < 1$, $c > 1$ and $c=1$. \\ 
Assume $c < 1$. 
Since $\mfat(z)$ is the Stieltjes Transform of a probability measure supported 
by $\R_+$, the function $\mfat(x)$ decreases to zero as $x\to-\infty$. 
Hence $x(\mfat) \to - \infty$ as $\mfat \to 0^+$.
From \eqref{x(m)} we notice that $\mfat x(\mfat) \to (1-c) / c > 0$ as $\mfat\to\infty$, 
hence $x(\mfat)$ reaches a positive maximum on $(0, \infty)$. By Propositions 
\ref{support} and \ref{disjoint}, this maximum is $A$, and we have 
$x < A \Rightarrow \mfat(x) > 0$. Therefore, $x_0 < A \Rightarrow 
\nu(E) = 0$. \\ 
Consider now the case $c > 1$. We shall also show that $x_0 < A \Rightarrow 
\nu(E) = 0$. 
By Proposition \ref{weight}, the measure $\mu$ has a Dirac at zero with
weight $1-c^{-1}$. Hence, either $x_0 < 0$, or $A > 0$ and $0 < x_0 < A$. 
Since $\mfat(z)$ is the Stieltjes Transform of a probability measure supported 
by $\R_+$, it holds that $x < 0 \Rightarrow \mfat(x) > 0$. Hence, 
$\nu(E) = 0$ when $x_0 < 0$. We now consider the second case. Since 
$(0, x_0] \subset \support(\mu)^c$, the image of this interval by $\mfat$
belongs to $B$. By Proposition \ref{weight}, $\lim_{x\to 0^+} \mfat(x) = 
-\infty$. Hence this image is $(-\infty, \mfat_0]$. This implies that 
$\nu(E) = 0$. \\
We finally consider the case $c=1$. We show here that $A=0$, which will 
result in $x_0 < A \Rightarrow \mfat_0 > 0 \Rightarrow \nu(E) = 0$ as above. 
Assume $A > 0$ and let $x_1 \in (0,A)$. By Proposition \ref{weight}, 
$\mu(\{0\}) = 0$ hence $\mfat(x_1) = \int(t-x_1)^{-1} \mu(dt) > 0$. But from 
\eqref{x(m)}, we easily show that $x(\mfat)$ increases from $-\infty$ to $0$ as 
$\mfat$ increases from $0$ to $\infty$, which raises a contradiction. 
This concludes the proof of Lemma \ref{m(1+ctm)}. 
\end{proof} 

 \begin{figure}[t]
  \begin{center}
    \includegraphics[width=0.7\linewidth]{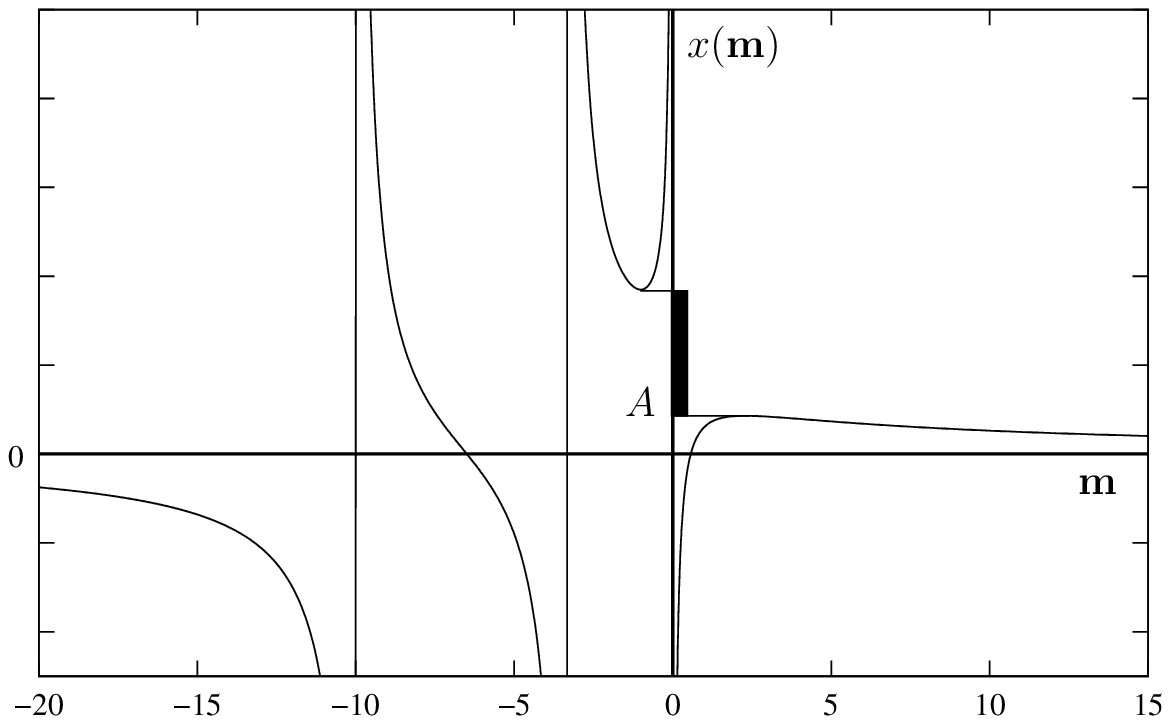}
  \end{center}
  \caption{Plot of $x(\mfat)$ for $c=0.1$ and 
    $\nu = 0.5(\bs\delta_1 + \bs\delta_3)$. The thick segment represents 
$\support(\mu)$.} 
  \label{fig:c<1}
\end{figure}
\begin{figure}[t]
  \begin{center}
    \includegraphics[width=0.7\linewidth]{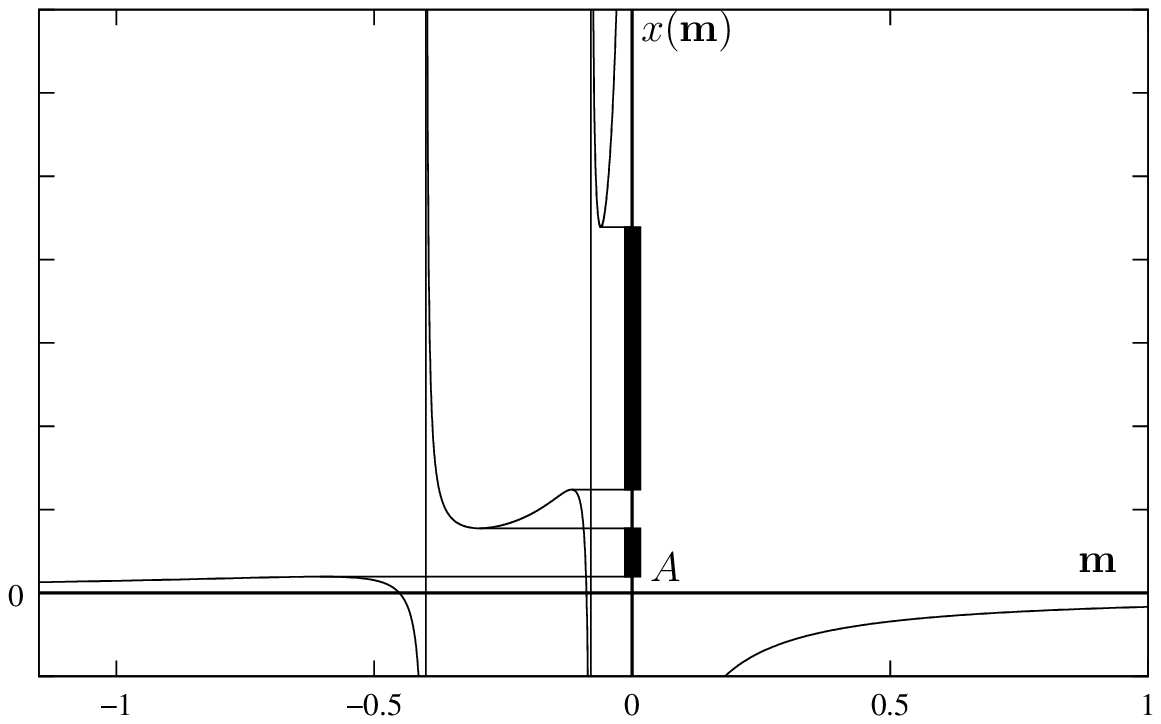}
  \end{center}
  \caption{Plot of $x(\mfat)$ for $c=5$ and 
            $\nu = 0.5(\bs\delta_{1/2} + \bs\delta_{5/2})$. The thick
segments represent $\support(\mu)$.} 
  \label{fig:c>1}
\end{figure}

This lemma shows that for any $x < \inf(\support(\mu) - \{ 0 \})$, 
$H_*(x) \geq 0$, hence ${\cal D}(x) > 0$ for those $x$. This proves Theorem 
\ref{1st-ord}-\eqref{no-spikes}. \\ 
Turning to Theorem \ref{partout}, let $E$ be a Borel set associated to
$(a,b)$ by Lemma \ref{m(1+ctm)}-$(ii)$ and let $q(x)$ be the function
defined in the statement of that lemma. The argument preceding Theorem
\ref{1st-ord} shows that the extension of $q(x)$ to $\C_+$ is the Stieltjes 
Transform of a positive measure. It results that $q(x)$ is negative and 
increasing on $(a,b)$. Let 
${\bs \Omega} = \diag( \omega_1^2, \ldots, \omega_r^2)$ where 
$\omega_k^2 = - 1 / q(\rho_k)$. Then it is clear that the function 
$\cal D(x) = \det( q(x) \bs \Omega + I_r)$ has $r$ roots in $(a,b)$ 
which coincide with the $\rho_k$. Theorem \ref{partout} 
will be proven if we find a sequence of matrices $P_n$ for which 
$H_*(z) = q(z) \bs\Omega$, \emph{i.e.}, $\Lambda_*(dt) = \1_E(t) \, \bs \Omega
\, \nu(dt)$. \\ 
Rearrange the elements of $ D_n$ in such a way that all the 
$ d_j^n$ which belong to $E$ are in the top left corner of this matrix. 
Let $M_n = [ M^n_{ij} ]$ be a random $\lfloor n \nu(E) \rfloor \times r$ 
matrix with iid elements such that $\sqrt{n} M^n_{11}$ has mean zero and 
variance one. Let $Z_n$ be the $n \times r$ matrix obtained by adding 
$n - \lfloor n \nu(E) \rfloor$ rows of zeros below $M_n$. 
Then the law of large numbers shows in conjunction with a normal 
family theorem argument that there is a set of probability one over which 
$z m_n(z) Z_n^* \widetilde T_n(z) Z_n$ converges to $q(z) I_r$ uniformly 
on the compact subsets of $(a,b)$.  
Consequently, there exists a sequence of deterministic matrices $B_n$ such
that $z m_n(z) B_n^* \widetilde T_n(z) B_n \to q(z) I_r$ uniformly on
these compact subsets. Matrix $P_n = A_n B_n^*$ with 
$A_n = \begin{bmatrix} \bs \Omega^{1/2} \\ 0_{(N-r) \times r} \end{bmatrix}$ 
satisfies the required property. Theorem \ref{partout} is proven. 

\subsection*{Proof of Corollary \ref{1stord-sp}} 
Observe from Proposition \ref{equiv-deter} that
$\int (1 + c\mfat(z)t)^{-1} \nu(dt) = - z \lim (n^{-1} \tr 
\widetilde T_n(z)) = - c z \mfat(z) + 1-c$. Consequently, in this particular
case, $H_*(x)$ is unitarily equivalent to 
$- \mfat(x) \left( c x \mfat(x) - 1 + c \right) {\bs \Omega}$ on $(a,b)$.

\section{Proof of the second order result} 
\label{sec:2ord}

We start by briefly showing Proposition \ref{var-biais}. 

\subsection*{Proof of Proposition \ref{var-biais}} 
For any $i = 1,\ldots, p$, it is clear that $\mfat(\rho_i)^2>0$ and 
$\mfat'(\rho_i)>0$. An immediate calculus then gives
$\mfat'(\rho_i) {\bs\Delta}(\rho_i) = \mfat^2(\rho_i)$ 
which shows that ${\bs \Delta}(\rho_i)  > 0$. \\
To prove the second fact, we shall establish more generally that  
$\limsup \sqrt{n} \| H_n(\rho_i) + g(\rho_i) {\bs\Omega} \| < \infty$.
 Invoking 
Equation \eqref{m=f(m)} and its analogue 
$m_n(z) = (-z + \int t (1+c_n m_n(z)t)^{-1} \nu_n(dt))^{-1}$, taking the 
difference and doing some straightforward derivations, we get that 
$(m_n(\rho_i) - \mfat(\rho_i)) ({\bs\Delta}(\rho_i) + \varepsilon_1) = 
\varepsilon_2$ where $\varepsilon_1 \to 0$ and where $|\varepsilon_2 | 
\leq K / \sqrt{n}$ thanks to the first two items of Assumption \ref{fast-cvg}.  
Hence $|m_n(\rho_i) - \mfat(\rho_i)| \leq K/\sqrt{n}$. Now we have
\begin{align*} 
H_n(\rho_i) + g(\rho_i) {\bs\Omega} &= 
\int \left( \frac{m_n(\rho_i)}{1 + c_n m_n(\rho_i) t} 
- \frac{\mfat(\rho_i)}{1 + c \mfat(\rho_i) t} \right) \Lambda_n(dt)  \\ 
&
+ \int \frac{\mfat(\rho_i)}{1 + c \mfat(\rho_i) t} \Lambda_n(dt)  
- \int \frac{\mfat(\rho_i)}{1 + c \mfat(\rho_i) t} \nu(dt) \times {\bs\Omega} . 
\end{align*} 
which shows thanks to Assumption \ref{fast-cvg} that 
$\limsup \sqrt{n} \| H_n(\rho_i) + g(\rho_i) {\bs\Omega} \| < \infty$.

\subsection*{Proof of Theorem \ref{th-2ord}} 
In all the remainder of this section, we shall work on a sequence of 
factorizations $P_n = U_n R_n^*$ such that $\Lambda_n$ satisfies the third 
item of Assumption \ref{fast-cvg}. We also write 
$U_n = [ U_{1,n} \, \cdots \, U_{t,n} ]$ and 
$R_n = [ R_{1,n} \, \cdots \, R_{t,n} ]$ where 
$U_{i,n} \in \C^{N \times j_i}$ and $R_{i,n} \in \C^{n \times j_i}$. \\

We now enter the core of the proof Theorem \ref{th-2ord}. The following 
preliminary lemmas are proven in the appendix: 

\begin{lemma}
\label{lm:clt2} 
Let $s$ be a fixed integer, and let $Z_N = [Z_{ij}] $ be a $N \times s$ 
complex matrix with iid elements with independent ${\mathcal N}(0,1/2)$ real 
and imaginary parts. Let $\Upsilon_N = [\Upsilon_{ij}] $ be a deterministic 
Hermitian $N\times N$ matrix such that $\tr \Upsilon_N = 0$, and let 
$F_N = [F_{ij}] $ be a complex deterministic $N \times s$ matrix. Assume that 
$F_N^* F_N \to {\bs\varsigma}^2 I_s$, that
$\limsup_N \| \Upsilon_N \| < \infty$, and that 
$N^{-1} \tr \Upsilon_N^2 \to {\bs\sigma}^2$ as $N\to\infty$. 
Let $M$ be a $s\times s$ complex matrix with iid elements with independent 
${\mathcal N}(0,1/2)$ real and imaginary parts, and let $G$ be a $s\times s$
GUE matrix independent of $M$. Then 
\[
\left( N^{-1/2} Z_N^* \Upsilon_N Z_N, Z_N^* F_N  \right) 
\xrightarrow[N\to\infty]{{\mathcal D}} 
\left( {\bs\sigma} G , {\bs\varsigma} M \right).
\]
\end{lemma} 

\begin{lemma}
\label{lm-aQYYQa}
For $x \in \support(\mu)^c$,
\begin{align*} 
	&\E\left[\phi_n \tilde{u}_n^* \widetilde Q_n(x)(n^{-1} Y_n^* Y_n) 
	\widetilde Q_n(x) \tilde{u}_n \right] = c_n
	\frac{x^2 m_n(x)^2 \tilde{u}_n^* D_n \widetilde T_n^2(x) \tilde{u}_n}
{1-c_n x^2 m_n(x)^2 \frac 1n \tr D_n^2 \widetilde T_n^2(x)} 
+ \mathcal O(n^{-1})
\end{align*} 
and
\[
\var \left(\phi_n \tilde{u}_n^* \widetilde Q_n(x) (n^{-1} Y_n^* Y_n) 
\widetilde Q_n(x) \tilde{u}_n \right) \leq \frac Kn . 
\]
\end{lemma} 


\begin{lemma}
\label{le:sigman} 
For $i = 1,\ldots, p$, let $A_i$ be a deterministic Hermitian $j_i \times j_i$ 
matrix independent of $n$, where $p$ and the $j_i$ are as in the statement
of Theorem \ref{th-2ord}. For $i=1,\ldots, p$, let $M_{i,n}$ a $n \times j_i$ 
matrix such that $\sup_n \| M_{i,n} \| < \infty$. Then for any $t\in\R$, 
\begin{multline*} 
\E\Bigl[ \exp\Bigl( \imath \sqrt{n} t \sum_{i=1}^p 
\tr A_i 
M_{i,n}^* ( \widetilde Q_n(\rho_{i}) - \widetilde T_n(\rho_{i}) ) M_{i,n} 
\Bigr) \Bigr] = \exp(-t^2 \tilde\sigma_n^2 / 2) + {\mathcal O}(n^{-1/2}) 
\end{multline*} 
where
\begin{align*}
\tilde\sigma_n^2 = 
&\sum_{i,k=1}^p  c_n \rho_i \rho_k m_n(\rho_i) m_n(\rho_k)  \\ 
&\quad\quad\times\frac{\tr A_i M_{i,n}^* \widetilde T_n(\rho_i) D_n 
\widetilde T_n(\rho_k) M_{k,n} A_k 
M_{k,n}^* \widetilde T_n(\rho_k) D_n \widetilde T_n(\rho_i)
M_{i,n}}
{1 - c_n \rho_i\rho_k m_n(\rho_i) m_n(\rho_k) 
\frac1n\tr  D_n\widetilde T_n(\rho_i)
 D_n\widetilde T_n(\rho_k)} . 
\end{align*} 
\end{lemma} 

Replacing the $M_{i,n}$ with the blocks $R_{i,n}$ of $R_n$ in the statement of 
Lemma \ref{le:sigman} and observing that 
\[
R_n^* \widetilde T_n(\rho_i)  D_n \widetilde T_n(\rho_k) R_n = 
\int \frac{t}
{\rho_i\rho_k (1+c_n m_n(\rho_i) t) (1+c_n m_n(\rho_k) t)} 
\Lambda_n(dt) , 
\]
we obtain from the third item of Assumption \ref{fast-cvg} that 
$\tilde\sigma_n^2 \to 
\sum_{i=1}^p \bs{\tilde\sigma}_{i}^2 \tr A_i^2$ where 
\[
\bs{\tilde\sigma}_{i}^2 = 
\frac{c \omega_i^4}{\rho_i^2\bs\Delta(\rho_i)} 
\Bigl( \int \frac{ \mfat(\rho_i)t}{(1 + c{\bf m}(\rho_i) t)^2} \nu(dt) \Bigr)^2 . 
\]
Invoking the Cramer-Wold device, this means that the $p$-uple of random 
matrices 
\[ 
\sqrt{n} \left( 
R_{i,n}^* ( \widetilde Q_n(\rho_{i}) - \widetilde T_n(\rho_{i}) ) R_{i,n} 
\right)_{i=1}^p
\]
converges in distribution towards 
$(\bs{\tilde\sigma}_i \widetilde G_i)_{i=1}^p$ where 
$\widetilde G_1, \ldots, \widetilde G_p$ are independent GUE matrices 
with $\widetilde G_i \in \C^{j_i \times j_i}$. \\

Lemmas \ref{lm:clt2}--\ref{le:sigman} lead to the following result which 
plays a central role in the proof of Theorem \ref{th-2ord}: 

\begin{lemma} 
\label{le:3p}
Consider the $3p$-uple of random matrices 
\begin{align*} 
&L_n = \sqrt{n} \times \\
&\left( 
\frac{U_{i,n}^* Y_n\widetilde Q_n(\rho_i) R_{i,n}}{\sqrt{n}} ,U_{i,n}^* ( Q_n(\rho_{i}) - m_n(\rho_{i}) I_N ) U_{i,n}, \ 
R_{i,n}^* ( \widetilde Q_n(\rho_{i}) - \widetilde T_n(\rho_{i}) ) R_{i,n} 
\right)_{i=1}^p .  
\end{align*} 
Define the following quantities 
\begin{align*}
\bs\varsigma_i^2        &= \frac{\omega_i^2}{ {\bs\Delta}(\rho_i) } \int \frac{\mfat^2(\rho_i)t}{(1+c\mfat(\rho_i)t)^2}\nu(dt) \\
\bs\sigma_i^2           &= \frac{1}{ {\bs\Delta}(\rho_i) } \int \frac{\mfat^4(\rho_i)t^2}{(1+c\mfat(\rho_i)t)^2}\nu(dt) \\
\bs{\tilde\sigma}_{i}^2 &= \frac{c \omega_i^4}{\rho_i^2\bs\Delta(\rho_i)} \left( \int \frac{ \mfat(\rho_i)t}{(1 + c{\bf m}(\rho_i) t)^2} \nu(dt) \right)^2.
\end{align*}
Let $M_1, \ldots , M_p$ be random matrices such that 
$M_i \in \C^{j_i \times j_i}$ and has independent elements with independent 
${\cal N}(0,1/2)$ real and imaginary parts. 
Let $G_1, \widetilde G_1, \ldots, G_p, \widetilde G_p$ be GUE matrices 
such that $G_i, \widetilde G_i \in \C^{j_i \times j_i}$. Assume in addition 
that $M_1$, $G_1$, $\widetilde G_1$, $\ldots$, $M_p$, $G_p$, $\widetilde G_p$ are 
independent. Then 
\begin{align*}
	L_n \overset{\mathcal D}{\underset{n\to\infty}{\longrightarrow}} \left( \bs\varsigma_i M_i, \bs\sigma_i G_i, \bs{\tilde\sigma}_i \widetilde G_i\right)_{i=1}^p.
\end{align*}
\end{lemma}

\begin{proof}
Let $\alpha_n(\rho) = N^{-1} \tr Q_n(\rho)$. By Lemmas \ref{lm-var} and 
\ref{mean-tr(Q)}, 
$\sqrt{n} ( \alpha_n(\rho_i) - m_n(\rho_i) ) \toprobashort 0$. 
Therefore, we can replace the $m_n(\rho_i)$ in the expression of $L_n$ by
$\alpha_n(\rho_i)$, as we shall do in the rest of the proof. \\ 
Write $s = j_1 + \cdots + j_p$ and let $Z_n$ be a $N \times s$ complex matrix 
with iid elements with independent ${\mathcal N}(0, 1/2)$ real and imaginary
parts. Assume that $Z_n$ and $X_n$ are independent. 
Write $Z_n = \begin{bmatrix} Z_{1,n} \cdots Z_{p,n} \end{bmatrix}$ where the
block $Z_{i,n}$ is $N\times j_i$. 
Let $n^{-1/2} X_n = W_n \Delta_n \widetilde W_n^*$ be a singular value 
decomposition of $n^{-1/2} X_n$. By assumption \ref{X:gauss}, the square 
matrices $W_n$ and $\widetilde W_n$ are Haar distributed over their respective
unitary groups, and moreover, $W_n$, $\Delta_n$ and $\widetilde W_n$ are 
independent. Let
\begin{align*} 
&\overline{L}_n = \\
&\sqrt{n} \left( 
\frac{ U_{n}^* Y_n \widetilde Q_n(\rho_i) R_{i,n}}{\sqrt{n}},U_{n}^* ( Q_n(\rho_{i}) - \alpha_n(\rho_{i}) I_N ) U_{n}, \ 
R_{i,n}^* ( \widetilde Q_n(\rho_{i}) - \widetilde T_n(\rho_{i}) ) R_{i,n} 
\right)_{i=1}^p .  
\end{align*} 
We have 
\begin{align*} 
\overline{L}_n 
\stackrel{\cal L}{=} & \Bigl( 
\sqrt{N} (Z_n^* Z_n)^{-1/2} Z_n^* F_{i,n},  
N^{1/2} (Z_n^* Z_n)^{-1/2} Z_n^* \Upsilon_{i,n} Z_n (Z_n^* Z_n)^{-1/2} , \\  
& \sqrt{n} 
R_{i,n}^* ( \widetilde Q_n(\rho_{i}) - \widetilde T_n(\rho_{i}) ) R_{i,n} 
\Bigr)_{i=1}^p 
\end{align*} 
where 
$F_{i,n} = c_n^{-1/2} \Delta_n \widetilde W_n^* D_n^{1/2} 
\widetilde Q_n(\rho_i) R_{i,n}$ and 
$\Upsilon_{i,n} = c_n^{-1/2} 
( (\Delta_n \widetilde W_n^* D_n \widetilde W_n \Delta_n - \rho_i)^{-1} 
- \alpha_n(\rho_i) I_N)$.  We shall now show that the term 
$\sqrt{N} (Z_n^* Z_n)^{-1/2} Z_n^* F_{i,n}$ can be replaced with 
$Z_n^* F_{i,n}$. By the law of large numbers, we have  
$N^{-1} Z_n^* Z_n \toasshort I_s$. 
By the independence of $Z_n$ and $(\Delta_n, \widetilde W_n)$, we have
$\E[\tr Z_n^* F_{i,n} F_{i,n}^* Z_n \, |\, 
(\Delta_n \widetilde W_n) ]
= s c_n^{-1} \tr R_{i,n}^* \widetilde Q_n(\rho_i) (n^{-1} Y_n^* Y_n) 
\widetilde Q_n(\rho_i) R_{i,n}$ 
whose limit superior is bounded with probability one. Hence 
$Z_n^* F_{i,n}$ is tight, proving that the replacement can be done. \\
By deriving the variances of the elements of 
$N^{-1/2} Z_n^* \Upsilon_{i,n} Z_n$ with 
respect to the law of $Z_n$, and by recalling that 
$\limsup_n \| \Upsilon_{i,n} \|$ is bounded with probability one, we obtain 
that these elements are also tight.  It results that we can replace $L_n$ with 
\begin{align*} 
\underline{L}_n = \left( 
Z_{i,n}^* F_{i,n}, \
\frac{Z_{i,n}^* \Upsilon_{i,n} Z_{i,n}}{\sqrt{N}},
\sqrt{n} R_{i,n}^* ( \widetilde Q_n(\rho_{i}) - \widetilde T_n(\rho_{i}) ) 
 R_{i,n} \right)_{i=1}^p .  
\end{align*} 
For $i = 1,\ldots, p$, let $A_i$ and $B_i$ be deterministic
Hermitian $j_i \times j_i$ matrices and let $C_i$ be deterministic complex 
$j_i \times j_i$ matrices, all independent of $n$. The lemma will be 
established if we prove that 
\begin{align} 
&\E\left\{ \exp\Bigl( \imath \sqrt{n} t \sum_{i=1}^p 
\tr A_i 
R_{i,n}^* ( \widetilde Q_n(\rho_{i}) - \widetilde T_n(\rho_{i}) ) R_{i,n} 
\Bigr)\nonumber \right.\\
&\quad \left.\times \E\Bigl[ 
\exp\Bigl( \imath  t \sum_{i=1}^p 
N^{-1/2} \tr B_i Z_{i,n}^* \Upsilon_{i,n} Z_{i,n} + 
\Re( \tr C_i Z_{i,n}^* F_{i,n}  )  
\Bigr) \,\Big|\, (\Delta_n, \widetilde W_n ) \Bigr] 
\right\}  \nonumber\\  
&\xrightarrow[n\to\infty]{} 
\prod_{i=1}^p \exp(-t^2 ( \bs{\tilde\sigma}_{i}^2 \tr A_i^2 
+ \bs\sigma_{i}^2 \tr B_i^2 + \frac12 \bs\varsigma_{i}^2 \tr C_iC_i^*)/ 2) . 
\label{cvg-char-fct} 
\end{align} 
In addition to the boundedness of $\| \Upsilon_{i,n} \|$ w.p.~one, 
we have $\tr \Upsilon_{i,n} = 0$, and 
\begin{multline*} 
\frac 1N \tr \Upsilon_{i,n}^2 = 
\frac{1}{c_n N} \sum_{\ell=1}^N 
\left( (\lambda_\ell^n - \rho_i)^{-1} - \alpha_n(\rho_i) \right)^2 \\ 
\toaslong c^{-1} ( \mfat'(\rho_i) - \mfat(\rho_i)^2 ) = 
c^{-1} \mfat(\rho_i)^2 \left({\bs \Delta}(\rho_i)^{-1}-1\right) 
= {\bs\sigma}_i^2 . 
\end{multline*} 
Moreover, using Lemma \ref{lm-aQYYQa} in conjunction with Assumption 
\ref{ass:sp}, we obtain
\[
F_{i,n}^* F_{i,n} = c_n^{-1} 
R_{i,n}^* \left( \widetilde Q_n(\rho_i)\frac1nY_n^* Y_n\widetilde Q_n(\rho_i) 
\right) R_{i,n} \toprobalong \bs\varsigma_i^2 I_{j_i} . 
\]
From any sequence of integers increasing to infinity, there exists a
subsequence along which this convergence holds in the almost 
sure sense. Applying Lemma \ref{lm:clt2}, we get that the inner expectation
at the left hand side of \eqref{cvg-char-fct} converges almost surely
along this subsequence towards 
$\prod_{i=1}^p \exp(-t^2 
(\bs\sigma_{i}^2 \tr B_i^2 + \frac12 \bs\varsigma_{i}^2 \tr C_iC_i^*)/ 2)$. 
Using in addition Lemma \ref{le:sigman} along with the dominated convergence
theorem, we obtain that Convergence \eqref{cvg-char-fct} holds true
along this subsequence. Since the original sequence is arbitrary, we
obtain the required result.
\end{proof}


The remainder of the proof of Theorem~\ref{th-2ord} is an adaptation of the
approach of \cite{bgm-fluct10}. 

\begin{lemma}
\label{le:chi_n}
For a given $x \in \R$ and a given $i \in \{1,\ldots, p\}$, let 
$y = \rho_i + n^{-1/2} x$, and let 
\[
\widehat S_n(y) = 
\begin{bmatrix}
\sqrt{y} U_n^* Q_n(y) U_n & I_r + n^{-1/2} U_n^* Y_n \widetilde Q_n(y) R_n \\
I_r + n^{-1/2} R_n^* \widetilde Q_n(y) Y_n^* U_n & 
\sqrt{y} R_n^* \widetilde Q_n(y) R_n 
\end{bmatrix}  
\]
Let 
\begin{align}
\label{eq:chi_n}
&\chi^{(i)}_n(x) = n^{\frac{j_i}2}
\Bigg[\det \hat{S}_n(y) 
- \prod_{k\neq i}^t
\left[\omega_k^2  g(\rho_i) -1 \right]^{j_k} 
\nonumber \\ 
	&\times\det\Bigg( 
\frac{\sqrt{n} U_{i,n}^*\left( Q_n(\rho_i)-m_n(\rho_i)I_N\right)U_{i,n}}
{\mfat(\rho_i)} + 
\rho_i \mfat(\rho_i) \sqrt{n} R_{i,n}^* ( \tQ_n(\rho_i)-\tT_n(\rho_i))R_{i,n}
 \nonumber \\
	&- 2 \Re\left[U_{i,n}^* Y_n \widetilde{Q}_n(\rho_i) R_{i,n} \right] 
- \sqrt{n}(H_{i,n}(\rho_i)+I_{j_i}) - x H_{i,n}'(\rho_i) \Bigg) \Bigg] 
\end{align}
where $\Re(M) = (M + M^*)/2$ for a square matrix $M$. 
Then
\begin{align*}
	(\chi^{(i)}_n(x_1),\ldots,\chi^{(i)}_n(x_p))\toprobalong 0
\end{align*}
for every finite sequence $\{x_1,\ldots,x_p\}$.
\end{lemma}
\begin{proof}
We show the result for $i=1$, the same procedure being valid for the other 
values of $i$. The notation $X_n = o_P(1)$ means that the random variable
$X_n$ converges to zero in probability, while $X_n = {\cal O}_P(n^{-\ell})$ 
means that $n^\ell X_n$ is tight. 
Write $U = [ U_1, \bar U_1]$ and $R = [ R_1, \bar R_1 ]$ where
$\bar U_1 = [ U_2, \ldots, U_t]$ and $\bar R_1 = [ R_2, \ldots, R_t]$. Writing
\begin{align*} 
A &= \begin{bmatrix} \sqrt{y} U_1^* Q(y) U_1 &  \sqrt{y} U_1^* Q(y) \bar U_1 \\  
\sqrt{y} \bar U_1^* Q(y) U_1 &  \sqrt{y} \bar U_1^* Q(y) \bar U_1 \end{bmatrix}
:= \begin{bmatrix} A_{11} & A_{12} \\ A_{12}^* & A_{22} \end{bmatrix}, \\
B &= \begin{bmatrix} I_{j_1} + n^{-1/2} U_1^* Y \tQ(y) R_1 & 
                            n^{-1/2} U_1^* Y \tQ(y) \bar R_1  \\ 
n^{-1/2} \bar U_1^* Y \tQ(y) R_1 & 
            I_{r-j_1} + n^{-1/2} \bar U_1^* Y \tQ(y) \bar R_1 \end{bmatrix}
:= 
\begin{bmatrix} B_{11} & B_{12} \\ B_{21} & B_{22} \end{bmatrix}, \ \text{and}\\
C &= \begin{bmatrix} \sqrt{y} R_1^* \tQ(y) R_1 &  
\sqrt{y} R_1^* \tQ(y) \bar R_1 \\  
\sqrt{y} \bar R_1^* \tQ(y) R_1 &  
\sqrt{y} \bar R_1^* \tQ(y) \bar R_1 \end{bmatrix}
:= \begin{bmatrix} C_{11} & C_{12} \\ C_{12}^* & C_{22} \end{bmatrix}, 
\end{align*} 
we have 
\[
\det \widehat S = \det \begin{bmatrix} A & B \\ B^* & C \end{bmatrix} = 
\det \begin{bmatrix} 
A_{11} & B_{11} & & A_{12} & B_{12} \\ 
B_{11}^* & C_{11} & & B_{21}^* & C_{12} \\ \\ 
A_{12}^* & B_{21} & & A_{22} & B_{22} \\
B_{12}^* & C_{12}^* & & B_{22}^* & C_{22} 
\end{bmatrix} 
:= \det \begin{bmatrix} M_{11} & M_{12} \\ M_{12}^* & M_{22} \end{bmatrix} 
\] 
after a row and column permutation. Hence 
$
n^{j_1/2} \det \widehat S = 
\det M_{22} \times  
n^{j_1/2} \det( M_{11} - M_{12} M_{22}^{-1} M_{12}^* )  . 
$
Write ${\bs\Omega}=\diag(\omega_1^2I_{j_1},{\bs \Omega}_2)$. From the first 
order analysis we get that 
\[
M_{22} \toaslong 
\begin{bmatrix} \sqrt{\rho_1} \mfat(\rho_1) I_{r-j_1} &  I_{r-j_1} \\
I_{r-j_1} & 
\sqrt{\rho_1} ( c \mfat(\rho_1) - \rho_1^{-1} (1-c) ) {\bs \Omega}_2
\end{bmatrix} 
\]
which is invertible since $\det M_{22} \toasshort 
\prod_{k > 1} (\omega_k^2 g(\rho_1) - 1)^{j_k} \neq 0$. Moreover, 
$\| M_{12} \| = {\cal O}_P(n^{-1/2})$. To see this, consider for instance the 
term $\sqrt{n} C_{12} = \sqrt{ny} R_1^* (\tQ - \tT) \bar R_1 + 
\sqrt{ny} R_1^* \tT \bar R_1$. The first term is tight by Lemma 
\ref{le:3p} while the second is bounded by Assumption \ref{fast-cvg}. The
other terms are treated similarly. It results that 
$\| M_{12} M_{22}^{-1} M_{12}^* \| = {\cal O}_P(1/n)$. \\ 
In addition, 
$\det(y^{-1/2} C_{11}) \toasshort [ \omega_1^2 (c \mfat(\rho_1) - 
\rho_1^{-1} (1-c) )]^{j_1} = (\rho_1 \mfat(\rho_1))^{-j_1}$
by the definition of $\rho_1$. From these observations we get that
\begin{multline*} 
n^{j_1/2} \det \widehat S = 
\left( \prod_{k > 1} (\omega_k^2 g(\rho_1) - 1)^{j_k} + o_P(1) \right) 
\left( (\rho_1 \mfat(\rho_1))^{-j_1} + o_P(1) \right) \\ 
\times \det\left( \sqrt{ny} A_{11} - \sqrt{ny} B_{11} C_{11}^{-1} B_{11}^* + 
{\cal O}_P(n^{-1/2}) \right)  . 
\end{multline*} 
Now we make the expansion 
\begin{align}
\label{eq:determinant}
&\sqrt{ny} A_{11} - \sqrt{ny} B_{11} C_{11}^{-1} B_{11}^*  \\
&= y \sqrt{n} U_1^*(Q(y)-m_n(y)I_N) U_1+ y\sqrt{n} m_n(y)I_{j_1} \nonumber \\
&+ \sqrt{n}\left(I_{j_1}+ U_1^*\frac{Y\tQ(y)}{\sqrt{n}}R_1\right)(R_1^*\tQ(y) R_1)^{-1}(R_1^*(\tQ(y)-\tT(y))R_1)(R_1^*\tT(y) R_1)^{-1} \nonumber \\
&\times \left(I_{j_1}+ R_1^*\frac{\tQ(y) Y^*}{\sqrt{n}}U_1\right) \nonumber\\
&- \sqrt{n}\left(I_{j_1}+ U_1^*\frac{Y\tQ(y)}{\sqrt{n}}R_1\right)(R_1^*\tT(y) R_1)^{-1}\left(I_{j_1}+ R_1^*\frac{\tQ(y) Y^*}{\sqrt{n}}U_1\right). 
	\end{align}
To go further, remark that
\begin{align*}
	y\sqrt{n}& m_n(y)I_{j_1} - \sqrt{n}(R_1^*\tT(y) R_1)^{-1} \\
	&= (R_1^*\tT(y) R_1)^{-1} \left[\sqrt{n} \left( y m_n(y) R_1^*\tT(y) R_1 - \rho_1 m_n(\rho_1) R_1^*\tT(\rho_1) R_1 \right) \nonumber \right.\\ 
	& \left.+ \sqrt{n} \left(\rho_1 m_n(\rho_1) R_1^*\tT(\rho_1) R_1 - I_{j_1}\right)\right] \\
	&= \rho_1\mfat(\rho_1) 
\left[ - xH_1'(\rho_1) - \sqrt{n} (H_1(\rho_1)+I_{j_1})\right] 
+ o(1) 
\end{align*}
where we recall that $y=\rho_1+xn^{-\frac12}\to \rho_1$ and that
$R_1^*\tT(y) R_1\to (\rho_1 \mfat(\rho_1))^{-1} I_{j_1}$.
Recall from Lemma \ref{le:3p} that $\sqrt{n} U_1^* (Q - m_n I) U_1$, 
$\sqrt{n} R_1^* (\tQ - \tT) R_1$, and $U_1^* Y \tQ R_1$ are tight. Keeping 
the non negligible terms, we can write \eqref{eq:determinant} under the form
\begin{align*}
&\sqrt{ny} A_{11} - \sqrt{ny} B_{11} C_{11}^{-1} B_{11}^*  \\
	&= \rho_1 \sqrt{n} U_1^*(Q(\rho_1)-m_n(\rho_1)I_N) U_1 + 
(\rho_1 \mfat(\rho_1))^2 \sqrt{n}R_1^*(\tQ(\rho_1)-\tT(\rho_1)I_n)R_1\\
	&- 2\rho_1\mfat(\rho_1)\Re \left[U_1^*Y\tQ(\rho_1)R_1\right]
- \rho_1\mfat(\rho_1) 
\left( xH_1'(\rho_1) + \sqrt{n} (H_1(\rho_1)+I_{j_1})\right)
+ o_P(1).  
\end{align*}
Plugging this expression at the right hand side of the expression of 
$n^{j_1/2} \det \widehat S$ and observing that 
$H_{1,n}'(\rho_1) \to - \omega_1^2 g(\rho_1)' I_{j_1}$ concludes the proof. 
\end{proof}
For $i=1,\ldots,p$, take $x_1(i)>y_1(i)>x_2(i)>y_2(i)>\ldots>y_{j_i}(i)$ fixed sequences of real numbers. Call $J_n=(\sqrt{n}(\hat{\lambda}_{k(i)+\ell}^n-\rho_i),~i=1,\ldots,p,~\ell=1,\ldots,j_i)$, with $k(i)=\sum_{m=1}^{i-1} j_m$. Let also $C$ be the rectangle $C=[x_1(1),y_1(1)]\times \ldots\times [x_p(j_p),y_p(j_p)]$. Then, for all large $n$, we have
\begin{align*}
	\PP\left( J_n \in C \right) = \PP \left( \left\{\det \hat{S}_n\left(\rho_i + \frac{x_{\ell}(i)}{\sqrt{n}} \right)\det \hat{S}_n\left(\rho_i + \frac{y_{\ell}(i)}{\sqrt{n}} \right) < 0\right\}\right)
\end{align*}
since $\det\hat{S}_n(t)$ changes sign around $t=\hat{\lambda}_{k(i)+\ell}^n$, and only there (with probability one, for all large $n$).

From Lemma \ref{le:chi_n}, we see that, for growing $n$, the probability for the product of the determinants above to be negative for all $i$ and $\ell$ approaches the probability
\begin{align*}
	\PP\left( \left\{\det A_{x_{\ell}(i)} \det A_{y_{\ell}(i)}<0,~i=1,\ldots,p,~\ell=1,\ldots,j_i \right\}\right)
\end{align*}
where $A_x$ is the matrix
\begin{align*}
	A_x &= \frac{\sqrt{n} U_{i,n}^*
\left( Q_n(\rho_i)-m_n(\rho_i)I_N\right)U_{i,n}}{\mfat(\rho_i)} + 
\frac{\rho_i\sqrt{n} R_{i,n}^*( \tQ_n(\rho_i)-\tT_n(\rho_i))R_{i,n}}{\omega_i^2(c+c\rho_i\mfat(\rho_i)-1)} \\ 
	&- 2\Re\left[U_{i,n}^*Y\tilde{Q}_n(\rho_i)R_{i,n} \right] - \sqrt{n}(H_{i,n}(\rho_i)+I_{j_i}) - x H_{i,n}'(\rho_i). 
\end{align*}

This last probability is equal to $\PP(\bar{J}_n\in C)$, where $\bar{J}_n$ 
is the vector obtained by stacking the $p$ vectors of decreasingly ordered 
eigenvalues of the matrices 
\begin{align*}
	&B_i= [H_{i,n}'(\rho_i)]^{-1} \left( \frac{\sqrt{n} U_{i,n}^*\left( Q_n(\rho_i)-m_n(\rho_i)I_N\right)U_{i,n}}{\mfat(\rho_i)}\right. \\ 
	&\left.+ \frac{\rho_i\sqrt{n} R_{i,n}^*
( \tQ_n(\rho_i)-\tT_n(\rho_i))R_{i,n}}{\omega_i^2(c+c\rho_i\mfat(\rho_i)-1)} - 2\Re\left[U_{i,n}^*Y\tilde{Q}_n(\rho_i)R_{i,n} \right] - \sqrt{n}(H_{i,n}(\rho_i)+I_{j_i}) \right).
\end{align*}


From Lemma \ref{le:3p}, $\{B_1,\ldots,B_t\}$ asymptotically behave as scaled non-zero mean GUE matrices. Precisely, denoting $\bar{B}_i=H_{i,n}'(\rho_i) B_i + \sqrt{n}(H_{i,n}(\rho_i)+I_{j_i})$, from Lemma \ref{le:3p} and for all $a,b$,
\begin{align*}
	&\EE\left[ |(\bar{B}_i)_{ab}|^2\right] \\ 
	&\to \frac{ {\bs \sigma}_i^2 }{\mfat(\rho_i)^2} + \frac{\rho_i^2 {\bs{\tilde\sigma}}_i^2 }{\omega_i^4(c+c\rho_i\mfat(\rho_i)-1)^2} + 2{\bs\varsigma}^2 \\
	&= \frac{ {\bs \sigma}_i^2 }{\mfat(\rho_i)^2} + \rho_i^2 \mfat(\rho_i)^2 {\bs{\tilde\sigma}}_i^2 + 2{\bs\varsigma}^2 \\
	&= \frac{\mfat^2(\rho_i)}{ {\bs\Delta}(\rho_i) } 
\left[ \int \frac{t^2\nu(dt)}{(1+c\mfat(\rho_i)t)^2} 
+ c \omega_i^4 \left( \int 
\frac{ \mfat(\rho_i)t\nu(dt)}{(1 + c{\bf m}(\rho_i) t)^2} \right)^2 
+  \int \frac{2 \omega_i^2t\nu(dt)}{(1+c\mfat(\rho_i)t)^2} \right].
\end{align*}

This concludes the proof of Theorem~\ref{th-2ord}.


\appendix

\section{Proofs of Lemmas \ref{lm:clt2} to \ref{le:sigman}} 

\subsection{Proof of Lemma \ref{lm:clt2}} 
Given a $s \times s$ deterministic Hermitian matrix $A$ and a $s \times s$ 
deterministic complex matrix $B$, let 
$\Gamma_N = N^{-1/2} \Tr A Z_N^* \Upsilon_N Z_N + \Re ( \tr B Z_N^* F_N )$
where $\Re(M) = (M + M^*) / 2$ for any square matrix $M$. 
We shall show that for any $t \in \R$, 
\[ 
\varphi_N(t) := \E[\exp(\imath t \Gamma_N )] 
\xrightarrow[N\to\infty]{} \exp\Bigl( - t^2 \frac{{\bs\sigma}^2 \tr A^2 + 
{\bs\varsigma}^2 \tr BB^*  / 2}{2} \Bigr) 
:= \exp\Bigl(- \frac{t^2{\bs v}^2}{2}\Bigr) .  
\] 
The result will follow by invoking the Cram\'er-Wold device. 
To establish this convergence, we show that the derivative $\varphi_N'(t)$ 
satisfies 
$\varphi_N'(t) = - t {\bs v}^2 \varphi_N(t) + \varepsilon_N(t)$ where 
$\varepsilon_N(t) \to 0$ as $N\to\infty$ uniformly on any compact interval 
of $\R$. That being true, the function 
$\psi_N(t) = \varphi_N(t) \exp( t^2 {\bs v}^2 / 2)$ satisfies 
$\psi_N(t) = 1 + \int_0^t \varepsilon_N(u) \exp( u^2 {\bs v}^2 / 2) du 
\to 1$ which proves the lemma.  \\
By the IP formula, we get 
\begin{align*}
\varphi'(t) &= \imath \E [ \Gamma \exp(\imath t \Gamma) ] \\
&= \imath \E \Bigl[
\Bigl( 
\sum_{i,j=1}^s\sum_{k,\ell=1}^N \frac{ A_{ij} Z^*_{kj} \Upsilon_{k\ell} Z_{\ell i}}
{\sqrt{N}} + 
\sum_{i,j=1}^s \sum_{k=1}^N 
\frac{B_{ij} Z^*_{kj} F_{ki} + F_{ki}^* Z_{kj} B_{ij}^*}{2} 
\Bigr)  \\
& 
\ \ \ \ \ \ \ \ \ \ \ \ \ \ \ \ \ \ \ \ \ \ \ \ \ \ \ \ \ \ \ \ 
\ \ \ \ \ \ \ \ \ \ \ \ \ \ \ \ \ \ \ \ \ \ \ \ \ \ \ \ \ \ \ \ 
\ \ \ \ \ \ \ \ \ \ \ \ \ \ \ \ \ \ \ 
\times  \exp(\imath t \Gamma) \Bigr] \\
&= \imath \E \Bigl[
\sum_{i,j,k,\ell} \frac{A_{ij} \Upsilon_{k\ell}}{\sqrt N} 
\frac{\partial( Z_{\ell i} \exp(\imath t \Gamma) )}{\partial Z_{kj}} \\
& 
\ \ \ \ \ \ \ \ \ \ \ \ \ \ \ \ \ \ \ \ \ \ \ \ \ \ \ \ \ \ \ \ 
+ \frac 12 \sum_{i,j,k} 
B_{ij} F_{ki} \frac{\partial \exp(\imath t \Gamma)}{\partial Z_{kj}} 
+ 
F_{ki}^* B_{ij}^* \frac{\partial \exp(\imath t \Gamma)}{\partial Z_{kj}^*} 
\Bigr] . 
\end{align*} 
We obtain after a small calculation
\begin{gather*} 
\frac{\partial \exp(\imath t \Gamma)}{\partial Z_{kj}} =
\imath t \Bigl( 
\frac{[A Z^* \Upsilon]_{jk}}{\sqrt{N}} + \frac 12 [ B^* F^* ]_{jk} \Bigr) 
\exp(\imath t \Gamma) , \\ 
\frac{\partial \exp(\imath t \Gamma)}{\partial Z_{kj}^*} =
\imath t \Bigl( 
\frac{[\Upsilon Z A]_{kj}}{\sqrt{N}} + \frac 12 [ F B ]_{kj} \Bigr) 
\exp(\imath t \Gamma) 
\end{gather*} 
which leads to 
\begin{align*} 
\varphi'(t) &= 
  - t \E [ N^{-1} \tr A^2 Z^* \Upsilon^2 Z \, \exp(\imath t \Gamma) ] 
- (t/2) \tr(BB^* F^* F)  \, \varphi(t) \\ 
&\phantom{=} 
+ \imath N^{-1/2} \tr A \tr \Upsilon  \, \varphi(t)  \\
&\phantom{=} 
-t \E [ N^{-1/2} \tr A B^* F^* \Upsilon Z \, \exp(\imath t \Gamma) ]  
-(t/2) \E [ N^{-1/2} \tr Z^* \Upsilon FBA \, \exp(\imath t \Gamma) ] . 
\end{align*} 
Let us consider the first term at the right hand side of this equation. 
We have $\E[N^{-1} \tr A^2 Z^* \Upsilon^2 Z] = N^{-1} \tr A^2 \tr \Upsilon^2$. 
Applying the Poincar\'e-Nash inequality, we obtain after some calculations
that $\var( N^{-1} \tr A^2 Z^* \Upsilon^2 Z ) \leq 2N^{-2} \tr A^4 \tr \Upsilon^4 = 
{\cal O}(N^{-1})$ since $\| \Upsilon \|$ is bounded. It results that 
$\E [ N^{-1} \tr A^2 Z^* \Upsilon^2 Z \, \exp(\imath t \Gamma) ] = 
N^{-1} \tr A^2 \tr \Upsilon^2 \, \varphi(t) + {\cal O}(N^{-1/2})$
by Cauchy-Schwarz inequality. The third term is zero by hypothesis. 
Finally, 
$N^{-1} \E | \tr Z^* \Upsilon FBA |^2 = N^{-1} \tr \Upsilon^2 F B A^2 B^* F^*   
\leq N^{-1} \| \Upsilon \|^2 \tr F B A^2 B^* F^* = {\cal O}(N^{-1})$. 
Hence, the last two terms are ${\cal O}(N^{-1/2})$ by Cauchy-Schwarz 
inequality, which proves Lemma \ref{lm:clt2}. 

\subsection{An intermediate result.}

The following lemma will be needed in the proof of Lemma \ref{lm-aQYYQa}: 

\begin{lemma}
	\label{le:aQDQb}
	For $x,y\in {\rm supp}(\mu)^c$,
\begin{align*}
	\EE\left[\phi_n \frac1n\tr \tQ_n(x)\tD\tQ_n(y)\tD \right] &= \frac{\frac1n \tr \tD\tT_n(x)\tD\tT_n(y)}{1-c_n x m_n(x) ym_n(y)\frac1n\tr \tD\tT_n(x)\tD\tT_n(y)}  + \mathcal O(n^{-1})\\
	\EE\left[\phi_n \tilde{u}_n^*\tQ_n(x)\tD\tQ_n(y)\tilde{v}_n \right] &= \frac{\tilde{u}_n^*\tT_n(x)\tD \tT_n(y) \tilde{v}_n}{1-c_n x m_n(x) ym_n(y)\frac1n\tr \tD\tT_n(x)\tD\tT_n(y)} + \mathcal O(n^{-1}).
\end{align*}
\end{lemma}
\begin{proof}
We denote here $\tQ_x=\tQ(x)$ and drop all unnecessary indices. Using the IP
formula, we obtain
\begin{align*}
	&\EE\left[\phi Y_{ia}^*Y_{ij}\tQ_{x,jp}d_p\tQ_{y,pq} \right] = \frac{d_pd_j}n \left( \delta(a-j) \EE\left[\phi \tQ_{x,jp}\tQ_{y,pq} \right] \right. \\
	&\left. - \frac1n \EE\left[\phi \tQ_{x,jj}[Y\tQ_x]_{ip}Y_{ia}^* \tQ_{y,pq} \right] - \frac1n \EE\left[\phi Y_{ia}^*\tQ_{x,jp}\tQ_{y,pj}[Y\tQ_{y}]_{iq} \right] \right. \\
	&\left. + \EE\left[\frac1n[\adjugate(\psi)\psi'Y]_{ij}Y_{ia}^*\tQ_{x,jp}\tQ_{y,pq} \right] \right).
\end{align*}

Making the sum over $i$, $p$, and $j$, this is
\begin{align*}
	&\frac1n\EE\left[\phi [Y^*Y\tQ_xD\tQ_y]_{aq} \right] = \frac1{n^2}\EE\left[ [Y^*\adjugate(\psi)\psi'YD\tQ_xD\tQ_y]_{aq}\right] + c_n d_a \EE\left[\phi [\tQ_xD\tQ_y]_{aq} \right] \\
	&- \frac1n\EE\left[\phi \frac1n\tr D\tQ_x [Y^*Y\tQ_{x}D\tQ_y]_{aq} \right] - \frac1n\EE\left[\phi \frac1n\tr \tQ_x D\tQ_yD [Y^*Y\tQ_y]_{aq} \right]. 
\end{align*}

Using the relation $\frac1nY^*Y\tQ_x=x\tQ_x+I_n$ and appropriately gathering the terms on each side gives
\begin{align}
	\label{eq:aQDQb_1}
	&\EE\left[\phi [\tQ_x\tD\tQ_y]_{aq}(x-c_n\ d_a + x\frac1n\tr \tD\tQ_x) \right] \nonumber \\ 
	&= -\EE\left[\phi [\tD\tQ_y]_{aq}(1+\frac1n\tr \tD\tQ_x) \right] - \EE\left[\phi \frac1n\tr \tQ_x\tD\tQ_y\tD(\delta(a-q)+y[\tQ_y]_{aq}) \right] \nonumber \\
	&+ \EE\left[\frac1{n^2}[Y^*\adjugate(\psi)\psi'Y\tD\tQ_x\tD\tQ_y]_{aq} \right].
\end{align}
Introducing the term $\tilde{\beta}_x=\phi\frac1n\tr \tD\tQ_x$ and $\hat{\tilde{\beta}}_x=\tilde{\beta}_x-\phi\EE[\tilde{\beta}_x]$, we have
\begin{align}
	\label{eq:trace_quad}
	&\EE\left[\phi [\tQ_x\tD\tQ_y]_{aq}\right](x-c_n\ d_a + x \EE[\tilde{\beta}_x]) \nonumber \\
	&= - \EE\left[\phi [\tD\tQ_y]_{aq} \right](1+\EE[\beta_x]) - \EE\left[\phi \frac1n\tr \tQ_x\tD\tQ_y\tD \right](\delta(a-q)+y\EE[[\tQ_y]_{aq}]) \nonumber \\
	&-\EE\left[ [\tD\tQ_y]_{aq} \hat{\tilde{\beta}}_x\right] 
	- \EE\left[   [\tQ_x D\tQ_y]_{aq} x\hat{\tilde\beta}_x\right]
	- \EE\left[\phi \frac1n\tr \tQ_x\tD\tQ_y\tD y ([\tQ_y]_{aq}-\EE[\tQ_y]_{aq}]) \right] \nonumber \\
	&+\EE\left[\frac1{n^2}[Y^*\adjugate(\psi)\psi'Y\tD\tQ_x\tD\tQ_y]_{aq} \right].
\end{align}

At this point, we can prove both results for the trace and for the quadratic form. We start by dividing each side by $x-c_n\ d_a + x \EE[\tilde{\beta}_x]$. We begin with the trace result. Multiplying the resulting left- and right-hand sides by $\ d_a$, summing over $a=q$ and normalizing by $1/n$, we obtain
\begin{align*}
	\EE\left[\phi \frac1n\tr \tQ_x\tD\tQ_x\tD \right] &= -(1+\EE[\tilde{\beta}_x])\EE\left[\phi \frac1n\tr \tD\tQ_y\tD A_x\right] \\
	&- \EE\left[\phi \frac1n\tr \tQ_x\tD\tQ_y\tD\right] \left(y \EE[\phi \frac1n\tr\tD A_x \tQ_y] + \frac1n\tr \tD A_x \right) + \varepsilon_n
\end{align*}
where we denoted $A_x = (x (1+\EE[\tilde{\beta}_x]) I_n - c_n\tD)^{-1}$ and where
\begin{align}
	\label{eq:epsilon}
	\varepsilon_n &=  \EE\left[\tr \frac{Y^*\adjugate(\psi)\psi'Y}{n^3}\tD\tQ_x\tD\tQ_y\tD A_x\right]-\EE\left[ \frac1n\tr \tD\tQ_y\tD A_x \hat{\tilde{\beta}}_x\right]\nonumber \\ 
	&- \EE\left[\frac{1}{n}\tr \tQ_x D \tQ_y D A_x \hat{\tilde\beta}_x  \right]\nonumber \\
&	- \EE\left[\phi \frac1n\tr \tQ_x\tD\tQ_y\tD y \left(\frac{1}{n}\tr \tQ_y D A_x -\EE(\frac{1}{n}\tr \tQ_y D A_x)\right) \right].
\end{align}
From Lemma~\ref{mean-tr(Q)}, $\EE[\tilde{\beta}_x]=\tilde{\delta}_x+\mathcal O(n^{-2})$, where we denoted $\tilde{\delta}_x=\frac1n\tr \tD\tT_x$. Also, it is easily observed that
\begin{align}
	\label{eq:tTtD}
	\left(I_n (1+\tilde{\delta}_x)x - c_n \tD \right)^{-1} = -\frac1{1+\tilde{\delta}_x} \tT_x
\end{align}
with $\tT_x=\tT(x)$. Therefore, along with Lemma~\ref{mean-tr(Q)}, we now have
\begin{align*}
	&\EE\left[\phi \frac1n\tr \tQ_x\tD\tQ_x\tD \right] \\
	&= \frac1n \tr \tD\tT_x\tD\tT_y + \EE\left[\phi \frac1n\tr \tQ_x\tD\tQ_x\tD\right]\frac{y\frac1n\tr\tD\tT_x\tT_y +\frac1n\tr\tD\tT_x}{1+\tilde{\delta}_x} + \varepsilon_n + \mathcal O(n^{-2}).
\end{align*}
Using now the fact that $y\tT_y + I_n = c_n\frac1{1+\tilde{\delta}_y}\tD\tT_y$, we conclude
\begin{align*}
	\EE\left[\phi \frac1n\tr \tQ_x\tD\tQ_x\tD \right] &= \frac{\frac1n\tr \tD\tT_x\tD\tT_y}{1-c_n(1+\tilde{\delta}_x)^{-1}(1+\tilde{\delta}_y)^{-1}\frac1n\tr \tD\tT_x\tD\tT_y} + \varepsilon_n + \mathcal O(n^{-2}).
\end{align*}
It therefore remains to prove that $\varepsilon_n=\mathcal O(n^{-1})$. Due to the presence of $\phi$ in the expression of $\hat{\tilde\beta}_x$, and using Lemma~\ref{lm-var} and Cauchy-Schwarz inequality, one can see that the last three terms in the expression of $\varepsilon_n$ are $\mathcal O(n^{-1})$. 
As for the first term, it is treated in a similar manner as in the proof of Lemma~\ref{mean-fq}, and is $\mathcal O(n^{-2})$.

In order to prove the result on the quadratic form, we start again from \eqref{eq:trace_quad}. Dividing each side again by $x-c_n\ d_a + x \EE[\tilde{\beta}_x]$, introducing $[\tilde{u}]_a$, $[\tilde{v}]_q$, and summing over the indices, we obtain
\begin{align}
	\label{eq:aQDQb_2}
	&\EE\left[\phi \tilde{u}^* \tQ_x\tD\tQ_y \tilde{v} \right] \nonumber\\
	&= -\EE\left[\phi \tilde{u}^* A_x \tD \tQ_y \tilde{v} \right] - \EE\left[\phi \frac1n\tr \tQ_x\tD\tQ_y \tD \right]\left(\tilde{u}^*A_x(y\EE[\phi\tQ_y]+I_n)\tilde{v} \right) + \varepsilon'_n
\end{align}
where $\varepsilon_n'$ is very similar to $\varepsilon_n$ and is shown to be $\mathcal O(n^{-1})$ with the same line of arguments. Using Lemma~\ref{lm-var}, \eqref{eq:tTtD}, and the previous result on $\EE[\phi \frac1n\tr \tQ_x\tD\tQ_y \tD]$, we finally obtain
\begin{align*}
	&\EE\left[\phi \tilde{u}^* \tQ_x\tD\tQ_y \tilde{v} \right] \\
	&= \tilde{u}^*\tT_x\tD\tT_y\tilde{v}\left( 1+\frac{c_n(1+\tilde{\delta}_x)^{-1}(1+\tilde{\delta}_y)^{-1}\frac1n\tr\tD\tT_x\tD\tT_y}{1-c_n(1+\tilde{\delta}_x)^{-1}(1+\tilde{\delta}_y)^{-1}\frac1n\tr\tD\tT_x\tD\tT_y}\right) + \mathcal O(n^{-1}).
\end{align*}
from which
\begin{align*}
	\EE\left[\phi \tilde{u}^* \tQ_x\tD\tQ_y \tilde{v} \right] &= \frac{\tilde{u}^*\tT_x\tD\tT_y\tilde{v}}{1-c_n(1+\tilde{\delta}_x)^{-1}(1+\tilde{\delta}_y)^{-1}\frac1n\tr\tD\tT_x\tD\tT_y} + \mathcal O(n^{-1}).
\end{align*}
We conclude with the remark $xm_n(x)=-(1+\tilde{\delta}_x)^{-1}$.
\end{proof}

%
%

\subsection{Proof of Lemma \ref{lm-aQYYQa}} 
\label{prf-aQYYQa} 
The line of proof closely follows the proof of Lemma~\ref{le:aQDQb}. We provide
here its main steps. By the IP formula, we have
\begin{align*} 
\EE[ \phi \tQ_{pk} Y_{\ell k}^* Y_{\ell m} \tQ_{mr} ] &= 
-\frac{d_m}{n} \EE[\phi \tQ_{pk} Y_{\ell k}^* \tQ_{mm} [ Y\tQ]_{\ell r}] 
+ \delta(k-m) d_m \EE[ \phi \tQ_{pk} \tQ_{mr} ] \\ 
&\phantom{=} 
-\frac{d_m}{n} \EE[\phi Y_{\ell k}^* \tQ_{mr} \tQ_{pm} [ Y\tQ]_{\ell k}] \\
&\phantom{=} 
+ \frac{d_m}{n} \EE[ \tQ_{pk} Y_{\ell k}^* \tQ_{mr}
[\adjugate(\psi) \psi' Y]_{\ell m} ] 
\end{align*} 
Taking the sum over $m$, we obtain 
\begin{align*} 
\EE[ \phi \tQ_{pk} Y_{\ell k}^* [ Y \tQ]_{\ell r} ] &= 
  \frac{d_k}{1+\EE[\tilde\beta]} \EE[ \phi \tQ_{pk} \tQ_{kr} ] 
- \frac{1}{1+\EE\tilde\beta}
\frac{1}{n} \EE[\phi Y_{\ell k}^* [\tQ D \tQ]_{pr} [ Y\tQ]_{\ell k}] \\
&\phantom{=} 
+ \frac{1}{1+\EE[\tilde\beta]}
\frac{1}{n} \EE[ \tQ_{pk} Y_{\ell k}^* 
[\adjugate(\psi) \psi' Y D \tQ]_{\ell r} ] 
- \EE[\hat{\tilde\beta} \tQ_{pk} Y_{\ell k}^* [ Y\tQ]_{\ell r}] 
\end{align*} 
where $\tilde{\beta}(x) = \phi\frac1n\tr \tD\tQ(x)$ and 
$\hat{\tilde{\beta}}(x)=\tilde{\beta}(x)-\phi\EE[\tilde{\beta}(x)]$ as in the 
proof of Lemma~\ref{le:aQDQb}. Taking the sum over $\ell$ 
then over $k$, we obtain
\begin{align*} 
\EE[ \phi [ \tQ \frac{Y^* Y}{n} \tQ]_{p r} ] &= 
c_n   \frac{1}{1+\EE[\tilde\beta]} \EE[ \phi [ \tQ D \tQ ]_{pr} ] 
- \frac{1}{1+\EE[\tilde\beta]}
\EE[\phi [\tQ D \tQ]_{pr} \frac 1n \tr (\frac{Y^* Y}{n} \tQ) ] \\
&\phantom{=} 
+ \frac{1}{1+\EE[\tilde\beta]}
\frac{1}{n} \EE[ 
[\tQ \frac{Y^* \adjugate(\psi) \psi' Y}{n} D \tQ]_{p r} ] 
- \EE[\hat{\tilde\beta} [ \tQ \frac{Y^* Y}{n}\tQ]_{p r}] 
\end{align*} 
Observing that $(1 + \EE[\tilde\beta(x)])^{-1} = - x m_n(x) + {\cal O}(n^{-2})$
and making the usual approximations, we get 
\begin{align*} 
	\EE[ \phi \tilde{u}^* \tQ \frac{Y^* Y}{n} \tQ \tilde{u} ] &= 
\Bigl( x m_n(x) \frac 1n \tr (\EE[\phi \frac{Y^* Y}{n} \tQ]) 
- c_n x m_n(x) \Bigr)  
\EE[ \phi \tilde{u}^* \tQ D \tQ \tilde{u} ] + {\cal O}(n^{-1})  
\end{align*} 
Observing that $n^{-1} \tr (\EE[\phi (n^{-1}Y^* Y) \tQ(x)] = 
Nn^{-1} x m_n(x) + Nn^{-1} + {\cal O}(n^{-2})$ and invoking Lemma 
\ref{le:aQDQb}, we obtain the desired result.

\subsection{Proof of Lemma \ref{le:sigman}} 
\label{prf-sigman} 
As in the previous proofs, we discard unnecessary indices. We also denote $\tQ_i=\tQ(\rho_i)$. For readability, we also write $\tilde{M}_i=M_{i,n}A_i$ and use the shortcut notation $\Gamma=\sqrt{n}\sum_{i=1}^p\tr M_i^*\tQ_i\tilde{M}_i$. We focus first on the term in $\rho_1$. The line of proof closely follows that of Lemma~\ref{lm:clt2} with the exception that we need to introduce the regularization function $\phi$ to ensure the existence of all the quantities under study. That is, with $\varphi_N(t)=\EE[\exp(\imath t \phi\Gamma)]$, we only need to show that $\varphi_N'(t)=-t \tilde{\sigma}_n^2\varphi_N(t)+\mathcal O(1/\sqrt{n})$. Using $|\varphi_N(t)|\leq 1$ and Lemma~\ref{1-E[phi]}, $|\EE[\exp(\imath t\Gamma)]-\varphi_N(t)|\leq 1-\EE[\phi]\to 0$ as $N\to\infty$, from which the result unfolds.

Using the IP formula, we first obtain
\begin{align*}
	&\EE\left[\phi \left[\frac{Y^*Y}n\tQ_1\right]_{pq} e^{\imath t \phi \Gamma}\right] \\
	&= c_n \EE\left[\phi [\tD\tQ_1]_{pq}e^{\imath t \phi \Gamma}\right] - \EE\left[\phi \frac1n\tr \tD\tQ_1 \left[\frac{Y^*Y}n\tQ_1\right]_{pq}e^{\imath t \phi \Gamma} \right] \\
	&- \EE\left[\imath t e^{\imath t \phi \Gamma} \phi^2 \frac1{\sqrt{n}} \sum_{j=1}^p \sum_{a=1}^r \left[(\tilde{M}_j)^*_a \tQ_j\tD\tQ_1\right]_q \left[\frac{Y^*Y}n\tQ_j (M_j)_a\right]_p  \right] + \varepsilon_{n,pq}
\end{align*}
where 
\begin{align*}
	\varepsilon_{n,pq} &= \EE\left[\frac1n \left[\frac{Y^*\adjugate(\psi)\psi'Y}n\tD\tQ_1 \right]_{pq}e^{\imath t \phi \Gamma}\right] + \EE\left[\phi \frac1n \left[\frac{Y^*\adjugate(\psi)\psi'Y}n\tD\tQ_1 \right]_{pq}\imath t \Gamma e^{\imath t \phi \Gamma}\right]
\end{align*}
and where we denoted $X_a$ the column $a$ of matrix $X$, $X_a^*$ being the row vector $(X_a)^*$.

With $\tilde{\beta}_j=\phi \frac1n \tr D \tQ_j$, $\hat{\tilde\beta}_j=\tilde{\beta}_j-\phi \EE \left[\tilde{\beta}_j\right]$, and with the relation $n^{-1}Y^*Y\tQ_1=I_n+\rho_1\tQ_1$, we obtain
\begin{align*}
	&\left( \rho_1(1+\EE[\tilde{\beta}_1]) - c_n \ d_p \right) \EE\left[\phi [\tQ_1]_{pq} e^{\imath t \phi \Gamma} \right] = - \delta(p-q)(1+\EE[\tilde{\beta}_1]) \EE\left[\phi e^{\imath t \phi \Gamma} \right] \\
	&- \EE\left[\imath t e^{\imath t \phi \Gamma} \phi^2 \frac1{\sqrt{n}} \sum_{j=1}^p \sum_{a=1}^r \left[(\tilde{M}_j)^*_a \tQ_j\tD\tQ_1\right]_q \left[\frac{Y^*Y}n\tQ_j (M_j)_a\right]_p \right] + \varepsilon'_{n,pq}
\end{align*}
where
\begin{align*}
	\varepsilon'_{n,pq} &= \varepsilon_{n,pq} - \EE\left[ \hat{\tilde{\beta}}_1 \left[\frac{Y^*Y}n \tQ_1\right]_{pq} e^{\imath t \phi \Gamma}\right].
\end{align*}
Dividing each side by $\rho_1(1+\EE[\tilde{\beta}_1]) - c_n \ d_p$, then multiplying by $(\tilde{M}_1)_p$ and $(M_1)_q$, and summing over $p,q$ gives
\begin{align*}
	\EE[\phi \tr (\tilde{M}_1^*\tQ_1 M_1)e^{\imath t \phi \Gamma}] &= -(1+\EE[\tilde{\beta}_1])\EE[\phi e^{\imath t \Gamma}] \tr \left( \tilde{M}_1^*A_{\rho_1}M_1\right) \\
	&- \imath t \EE\left[\phi^2 e^{\imath t \phi\Gamma} \frac1{\sqrt{n}}\sum_{j=1}^p \tr \tilde{M}_1^* A_{\rho_1} \frac{Y^*Y}n\tQ_jM_j\tilde{M}_j^*\tQ_j\tD\tQ_1\right] + \varepsilon_n'
\end{align*}
with $A_{\rho_i}=( \rho_i(1+\EE[\tilde{\beta}_i])I_n-c_n \tD)^{-1}$, and
\begin{align*}
	\varepsilon_n' &= \tr \tilde{M}_1^* A_{\rho_1} E' M_1
\end{align*}
with $(E')_{pq}=\varepsilon_{pq}'$. From \eqref{eq:tTtD}, the identity $n^{-1}Y^*Y\tQ_j=I_n+\rho_j\tQ_j$, and Lemma~\ref{mean-tr(Q)}, we finally obtain
\begin{align*}
	&\EE\left[\phi \tr (\tilde{M}_1^*\tQ_1 M_1)e^{\imath t \phi \Gamma} \right] - \EE[\phi e^{\imath t \phi \Gamma}] \tr \tilde{M}_1\tT_1 M_1  \\
	&= \imath t \EE[\phi e^{\imath t \phi \Gamma}] \frac1{\sqrt{n}}\sum_{j=1}^p \frac{\tr \tilde{M}_1^* \frac{\tT_1}{1+\tilde{\delta}_1}\frac{c_n\tD\tT_j}{1+\tilde{\delta}_j} M_j \tilde{M}_j^*\tT_j\tD\tT_1 M_1}{1-c_n(1+\tilde{\delta}_1)^{-1}(1+\tilde{\delta_j})^{-1}\frac1n\tr \tD\tT_1\tD\tT_j} + \varepsilon_n' + \mathcal O(n^{-2}) 
\end{align*}
with $\tT_i=\tT(\rho_i)$, from which
\begin{align*}
	&\EE\left[\phi \tr (\tilde{M}_1^*\tQ_1 M_1)e^{\imath t \phi \Gamma} \right] - \EE[\phi e^{\imath t \phi \Gamma}] \tr \tilde{M}_1\tT_1 M_1 \\
	&= \frac{\imath t \EE[\phi e^{\imath t \phi \Gamma}]}{\sqrt{n}} \sum_{j=1}^p \frac{c_n \rho_1 m_n(\rho_1) \rho_j m_n(\rho_j) \tr \tilde{M}_1^* \tT_1\tD\tT_j M_j \tilde{M}_j^*\tT_j\tD\tT_1 M_1}{1 - c_n \rho_1 m_n(\rho_1) \rho_j m_n(\rho_j) \frac1n\tr \tD\tT_1\tD\tT_j} + \varepsilon_n' + \mathcal O(n^{-2}).
\end{align*}
It remains to show that $\varepsilon_n'=\mathcal O(n^{-1})$. We have explicitly
\begin{align*}
	\varepsilon_n' &= \EE\left[ \frac1n\tr \left(\tilde{M}_1^* A_{\rho_1} \frac{Y^*\adjugate(\psi)\psi'Y}n\tD\tQ_1 M_1\right) (1+\phi \imath t \Gamma)e^{\imath t\phi \Gamma} \right] \\
	&- \EE\left[\phi \hat{\tilde{\beta}}_1 \tr \left(\tilde{M}_1^* A_{\rho_1} \frac{Y^*Y}n \tQ_1 M_1\right) e^{\imath t \phi \Gamma}\right].
\end{align*}
Using the fact that $|e^{\imath t \phi \Gamma}|=1$ and the relation $n^{-1}Y^*Y\tQ_1=\rho_1\tQ_1+I_n$, the second term is easily shown to be $\mathcal O(n^{-1})$ from the Cauchy-Scwharz inequality and Lemma~\ref{lm-var}. If it were not for the factor $\Gamma$, the convergence of the first term would unfold from similar arguments as in the proof of Lemma~\ref{mean-fq}. We only need to show here that $\EE[|\phi\Gamma|^2]=\mathcal O(1)$. But this follows immediately from Lemma~\ref{lm-var} and Lemma~\ref{mean-tr(Q)}. \\
The generalization to $\sum_{i} \EE[\phi \tr (\tilde{M}_i^*\tQ_i M_i)e^{\imath t \phi \Gamma}]$ is then immediate and we have the expected result.


\begin{thebibliography}{10}

\bibitem{by08}
Z.~Bai and J.-F. Yao.
\newblock Central limit theorems for eigenvalues in a spiked population model.
\newblock {\em Ann. Inst. Henri Poincar\'e Probab. Stat.}, 44(3):447--474,
  2008.

\bibitem{Bai99}
Z.~D. Bai.
\newblock Methodologies in spectral analysis of large-dimensional random
  matrices, a review.
\newblock {\em Statist. Sinica}, 9(3):611--677, 1999.

\bibitem{BaiSil98}
Z.~D. Bai and J.~W. Silverstein.
\newblock No eigenvalues outside the support of the limiting spectral
  distribution of large-dimensional sample covariance matrices.
\newblock {\em Ann. Probab.}, 26(1):316--345, 1998.

\bibitem{bai-sil-separation}
Z.~D. Bai and Jack~W. Silverstein.
\newblock Exact separation of eigenvalues of large-dimensional sample
  covariance matrices.
\newblock {\em Ann. Probab.}, 27(3):1536--1555, 1999.

\bibitem{bbp05}
J.~Baik, G.~Ben~Arous, and S.~P{\'e}ch{\'e}.
\newblock Phase transition of the largest eigenvalue for nonnull complex sample
  covariance matrices.
\newblock {\em Ann. Probab.}, 33(5):1643--1697, 2005.

\bibitem{bk-sil06}
J.~Baik and J.W. Silverstein.
\newblock Eigenvalues of large sample covariance matrices of spiked population
  models.
\newblock {\em J. Multivariate Anal.}, 97(6):1382--1408, 2006.

\bibitem{bgm-fluct10}
F.~Benaych-Georges, A.~Guionnet, and M.~Ma{\"\i}da.
\newblock {Fluctuations of the extreme eigenvalues of finite rank deformations
  of random matrices}.
\newblock {\em Electron. J. Prob.}, 16:1621--1662, 2011.
\newblock Paper no. 60.

\bibitem{ben-rao-published}
F.~Benaych-Georges and R.~R. Nadakuditi.
\newblock {The eigenvalues and eigenvectors of finite, low rank perturbations
  of large random matrices}.
\newblock {\em Adv. in Math.}, 227(1):494--521, 2011.

\bibitem{ben-rao-11}
F.~Benaych-Georges and R.~R. Nadakuditi.
\newblock {The singular values and vectors of low rank perturbations of large
  rectangular random matrices}.
\newblock {\em J. Multivariate Anal.}, 111:295--309, 2012.

\bibitem{bolot97}
Vladimir Bolotnikov.
\newblock On a general moment problem on the half axis.
\newblock {\em Linear Algebra Appl.}, 255:57--112, 1997.

\bibitem{cdf09}
M.~Capitaine, C.~Donati-Martin, and D.~F{\'e}ral.
\newblock The largest eigenvalues of finite rank deformation of large {W}igner
  matrices: convergence and nonuniversality of the fluctuations.
\newblock {\em Ann. Probab.}, 37(1):1--47, 2009.

\bibitem{cdf-ihp12}
M.~Capitaine, C.~Donati-Martin, and D.~F{\'e}ral.
\newblock {Central limit theorems for eigenvalues of deformations of {W}igner
  matrices}.
\newblock {\em Ann. Inst. H. Poincar\'e Probab. Statist.}, 48(1):107--133,
  2012.

\bibitem{chen-jmva82}
L.~H.~Y. Chen.
\newblock An inequality for the multivariate normal distribution.
\newblock {\em J. Multivariate Anal.}, 12(2):306--315, 1982.

\bibitem{cou-hac-it12}
R.~Couillet and W.~Hachem.
\newblock Fluctuations of spiked random matrix models and failure diagnosis in
  sensor networks.
\newblock {\em IEEE Transactions on Information Theory}, 2012.
\newblock In press.

\bibitem{Gesztesy-herglotz}
Fritz Gesztesy and Eduard Tsekanovskii.
\newblock On matrix-valued {H}erglotz functions.
\newblock {\em Math. Nachr.}, 218:61--138, 2000.

\bibitem{Gir90}
V.~L. Girko.
\newblock {\em Theory of random determinants}, volume~45 of {\em Mathematics
  and its Applications (Soviet Series)}.
\newblock Kluwer Academic Publishers Group, Dordrecht, 1990.

\bibitem{haagerup}
U.~Haagerup and S.~Thorbj{\o}rnsen.
\newblock A new application of random matrices: {${\rm Ext}(C^*_{\rm
  red}(F_2))$} is not a group.
\newblock {\em Ann. of Math. (2)}, 162(2):711--775, 2005.

\bibitem{HLN08Ieee}
W.~Hachem, O.~Khorunzhy, P.~Loubaton, J.~Najim, and L.~Pastur.
\newblock A new approach for mutual information analysis of large dimensional
  multi-antenna chennels.
\newblock {\em IEEE Trans. Inform. Theory}, 54(9):3987--4004, 2008.

\bibitem{hlmnv-rmta11}
W.~Hachem, P.~Loubaton, X.~Mestre, J.~Najim, and P.~Vallet.
\newblock Large information plus noise random matrix models and consistent
  subspace estimation in large sensor networks.
\newblock {\em To be published in Random Matrices: Theory and Applications},
  2011.
\newblock ArXiv preprint arXiv:1106.5119.

\bibitem{hlmnv-subspace11}
W.~Hachem, P.~Loubaton, X.~Mestre, J.~Najim, and P.~Vallet.
\newblock A subspace estimator for fixed rank perturbations of large random
  matrices.
\newblock {\em ArXiv preprint arXiv:1106.1497}, 2011.

\bibitem{hlnv-bilinear-ihp11}
W.~Hachem, P.~Loubaton, J.~Najim, and P.~Vallet.
\newblock On bilinear forms based on the resolvent of large random matrices.
\newblock {\em To be published in Annales de l'IHP}, 2012.

\bibitem{hachem-loubaton-najim07}
W.~Hachem, Ph. Loubaton, and J.~Najim.
\newblock Deterministic equivalents for certain functionals of large random
  matrices.
\newblock {\em Ann. Appl. Probab.}, 17(3):875--930, 2007.

\bibitem{joh01}
I.~M. Johnstone.
\newblock On the distribution of the largest eigenvalue in principal components
  analysis.
\newblock {\em Ann. Statist.}, 29(2):295--327, 2001.

\bibitem{kho-pastur93}
A.~M. Khorunzhy and L.~A. Pastur.
\newblock Limits of infinite interaction radius, dimensionality and the number
  of components for random operators with off-diagonal randomness.
\newblock {\em Comm. Math. Phys.}, 153(3):605--646, 1993.

\bibitem{krein-nudel}
M.~G. Kre{\u\i}n and A.~A. Nudel{\cprime}man.
\newblock {\em The {M}arkov moment problem and extremal problems}.
\newblock American Mathematical Society, Providence, R.I., 1977.
\newblock Ideas and problems of P. L. {\v{C}}eby{\v{s}}ev and A. A. Markov and
  their further development, Translated from the Russian by D. Louvish,
  Translations of Mathematical Monographs, Vol. 50.

\bibitem{MarPas67}
V.~A. Mar{\v{c}}enko and L.~A. Pastur.
\newblock Distribution of eigenvalues in certain sets of random matrices.
\newblock {\em Mat. Sb. (N.S.)}, 72 (114):507--536, 1967.

\bibitem{mes-it08}
X.~Mestre.
\newblock Improved estimation of eigenvalues and eigenvectors of covariance
  matrices using their sample estimates.
\newblock {\em IEEE Trans. Inform. Theory}, 54(11):5113--5129, 2008.

\bibitem{mes-sp08}
X.~Mestre.
\newblock On the asymptotic behavior of the sample estimates of eigenvalues and
  eigenvectors of covariance matrices.
\newblock {\em IEEE Trans. Signal Process.}, 56(11):5353--5368, 2008.

\bibitem{chemo}
T.~N{\ae}s, T.~Isaksson, T.~Fearn, and T.~Davies.
\newblock {\em A User-Friendly Guide to Multivariate Calibration and
  Classification}.
\newblock NIR Publications, Chichester, UK, 2002.

\bibitem{pastur-etal95}
L.~A. Pastur, A.~M. Khorunzhi{\u\i}, and V.~Yu. Vasil{\cprime}chuk.
\newblock On an asymptotic property of the spectrum of the sum of
  one-dimensional independent random operators.
\newblock {\em Dopov. Nats. Akad. Nauk Ukra\"\i ni}, (2):27--30, 1995.

\bibitem{pas-livre}
L.~A. Pastur and M.~Shcherbina.
\newblock {\em Eigenvalue distribution of large random matrices}, volume 171 of
  {\em Mathematical Surveys and Monographs}.
\newblock American Mathematical Society, Providence, RI, 2011.

\bibitem{paul-07}
D.~Paul.
\newblock Asymptotics of sample eigenstructure for a large dimensional spiked
  covariance model.
\newblock {\em Statist. Sinica}, 17(4):1617--1642, 2007.

\bibitem{peche-06}
S.~P{\'e}ch{\'e}.
\newblock The largest eigenvalue of small rank perturbations of {H}ermitian
  random matrices.
\newblock {\em Probab. Theory Related Fields}, 134(1):127--173, 2006.

\bibitem{pot-bouch-lal-05}
M.~Potters, J.-P. Bouchaud, and L.~Laloux.
\newblock Financial applications of random matrix theory: Old laces and new
  pieces.
\newblock In {\em Proceedings of the Cracow conference on ``Applications of
  Random Matrix Theory to Economy and Other Complex Systems''}, 2005.
\newblock arXiv preprint: arXiv:physics/0507111.

\bibitem{Sil95}
J.~W. Silverstein.
\newblock Strong convergence of the empirical distribution of eigenvalues of
  large-dimensional random matrices.
\newblock {\em J. Multivariate Anal.}, 55(2):331--339, 1995.

\bibitem{SilBai95}
J.~W. Silverstein and Z.~D. Bai.
\newblock On the empirical distribution of eigenvalues of a class of
  large-dimensional random matrices.
\newblock {\em J. Multivariate Anal.}, 54(2):175--192, 1995.

\bibitem{sil-choi95}
Jack~W. Silverstein and Sang-Il Choi.
\newblock Analysis of the limiting spectral distribution of large-dimensional
  random matrices.
\newblock {\em J. Multivariate Anal.}, 54(2):295--309, 1995.

\bibitem{tse-vis-livre-05}
D.~Tse and P.~Viswanath.
\newblock {\em Fundamentals of wireless communication}.
\newblock Cambridge university press, 2005.

\bibitem{val-lou-mes-it12}
P.~Vallet, P.~Loubaton, and X.~Mestre.
\newblock Improved subspace estimation for multivariate observations of high
  dimension: The deterministic signals case.
\newblock {\em Information Theory, IEEE Transactions on}, 58(2):1043 --1068,
  feb. 2012.

\bibitem{vino-etal-ieeesp-sub12}
J.~Vinogradova, R.~Couillet, and W.~Hachem.
\newblock Statistical inference in large antenna arrays under unknown noise
  pattern.
\newblock {\em ArXiv preprint}, 2012.

\end{thebibliography}

\def\cprime{$'$} \def\cdprime{$''$} \def\cprime{$'$} \def\cprime{$'$}
  \def\cprime{$'$} \def\cprime{$'$}

\newpage 

\noindent {\sc Fran\c cois Chapon}, \\
Universit\'e Paul Sabatier, \\ 
Institut de Math\'ematiques de Toulouse, \\ 
118 route de Narbonne, \\ 
31062 Toulouse Cedex 9, \\ 
France. \\ 
e-mail: \texttt{francois.chapon@math.univ-toulouse.fr}\\
\\
\noindent {\sc Romain Couillet},\\
Sup\'elec, \\
Plateau de Moulon, \\
91192 Gif-sur-Yvette Cedex, \\
France. \\ 
e-mail: \texttt{romain.couillet@supelec.fr} \\
\\ 
\noindent {\sc Walid Hachem}, \\
CNRS, T\'el\'ecom Paristech\\
46, rue Barrault, \\
75634 Paris Cedex 13, \\
France.\\
e-mail: \texttt{walid.hachem@telecom-paristech.fr}\\
\\
\noindent {\sc Xavier Mestre},\\
Centre Tecnol\`ogic de Telecomunicacions de Catalunya (CTTC), \\
Parc Mediterrani de la Tecnologia (PMT) - Building B4, \\
Av.~Carl Friedrich Gauss 7, \\
08860 - Castelldefels \\
Barcelona, Spain. \\
e-mail: \texttt{xavier.mestre@cttc.cat} \\

\end{document}